\documentclass{amsart}
\usepackage{amsmath,amsthm}
\usepackage{hyperref}
\usepackage[psamsfonts]{amssymb}
\usepackage[all]{xy}
\usepackage{graphicx}
\usepackage[T1]{fontenc}
\usepackage[bitstream-charter]{mathdesign}
\newtheorem{theorem}{Theorem}[section]
\newtheorem{prop}[theorem]{Proposition}
\newtheorem{lemma}[theorem]{Lemma}
\newtheorem{cor}[theorem]{Corollary}
\newtheorem{add}[theorem]{Addendum}
\newtheorem{conj}[theorem]{Conjecture}

\newtheorem{ex}[theorem]{Example}
\theoremstyle{remark}
\newtheorem{dfn}[theorem]{Definition}

\newtheorem{remark}[theorem]{Remark}
\newtheorem{claim}[theorem]{Claim}

\def\co{\colon\thinspace}

\def\ep{\epsilon}

\def\R{\mathbb{R}}

\def\C{\mathbb{C}}

\title{Submanifolds and the Hofer norm}
\author{Michael Usher}
\date{\today}
\address{Department of Mathematics\\University of Georgia\\Athens, GA 30602}
\email{\href{mailto:usher@math.uga.edu}{usher@math.uga.edu}}
\begin{document}
\begin{abstract}
In \cite{Ch00}, Chekanov showed that the Hofer norm on the Hamiltonian diffeomorphism group of a geometrically bounded symplectic manifold induces a nondegenerate metric on the orbit of any compact Lagrangian submanifold under the group.  In this paper we consider the orbits of more general submanifolds.  We show that, for the Chekanov--Hofer pseudometric on the orbit of a closed submanifold to be a genuine metric, it is necessary for the submanifold to be coisotropic, and we show that this condition is sufficient under various additional geometric assumptions.  At the other extreme, we show that the image of a generic closed embedding with any codimension larger than one is ``weightless,'' in the sense that the Chekanov--Hofer pseudometric on its orbit vanishes identically.  In particular this yields examples of submanifolds which have zero displacement energy but are not infinitesimally displaceable.
\end{abstract}


%



\maketitle




\section{Introduction}  Since its introduction in \cite{Ho}, the \textbf{Hofer norm} $\|\cdot\|$ on the group $Ham(M,\omega)$ of (compactly supported) Hamiltonian diffeomorphisms 
of a symplectic manifold $(M,\omega)$ has been an important tool in the study of that group.  Of course, in attempting to understand a group, it is often useful to study natural actions of that group on various sets. Because the Hofer norm is invariant under inversion and conjugation, if $Ham(M,\omega)$ acts transitively on a set $S$ then we obtain a pseudometric $\delta$ on $S$ which is invariant under the action of $Ham(M,\omega)$, defined by \[ \delta(s_0,s_1)=\inf\left\{\|g\|\left|g\in Ham(M,\omega),\,gs_0=s_1 \right.\right\}   \]  The present paper studies this pseudometric in the case that $S$ is equal to the orbit $\mathcal{L}(N)$ of a closed subset $N\subset M$ under the action of $Ham(M,\omega)$.  The special case which has so far received the most attention is that in which $N$ is a compact Lagrangian submanifold\footnote{As indicated in Section \ref{conv}, in this paper all manifolds will be assumed to have no boundary unless the modifier ``with boundary'' is explicitly added.  We will avoid using the conventional term ``closed (sub)manifold'' to refer to a compact (sub)manifold without boundary, as we will sometimes consider submanifolds which are closed as subsets but which may not be compact, and it would be confusing to at the same time use the term ``closed submanifold'' to mean something other than this.}: Oh showed in \cite[p. 508]{Oh1} that $\delta$ defines a nondegenerate metric on $\mathcal{L}(N)$ when $N$ is the zero section of the cotangent bundle of a compact manifold, and shortly thereafter Chekanov \cite{Ch00} showed more generally that if $N$ is any compact Lagrangian submanifold of a geometrically bounded\footnote{The definition of ``geometrically bounded'' will be recalled in Section \ref{lagsec}.  Geometrically bounded symplectic manifolds include for instance those which are compact or convex (\emph{i.e.}, obtained as the Liouville completion of a compact manifold with contact type boundary), as well as products of these.  See \cite[Section 2]{CGK} for a proof that convex manifolds satisfy the property.} symplectic manifold then $\delta$ is nondegenerate on $\mathcal{L}(N)$. (Accordingly we will generally call $\delta$ the ``Chekanov--Hofer pseudometric,'' and say that $N$ is \textbf{CH-rigid} when $\delta$ is nondegenerate.)  See, \emph{e.g.}, \cite{Os}, \cite{Kh}, \cite{U2} for other results concerning the Chekanov--Hofer metric for Lagrangian submanifolds.

On the other hand, if one takes $N$ to be a singleton then it is quite easy to see that the Chekanov--Hofer pseudometric vanishes identically on $\mathcal{L}(N)$.  We will see in this paper that this continues to be the generic situation even when $N$ has relatively high dimension, though not when $N$ has codimension one.  Namely, where by definition a ``closed embedding'' is an embedding whose image is a closed subset, we have:

\begin{theorem}\label{int1} Let $(M,\omega)$ be a symplectic manifold and $X$ a connected smooth manifold. \begin{itemize} \item[(i)] If $\dim X=\dim M-1$, then for any closed embedding $f\co X\to M$ the image $f(X)$ is CH-rigid (\emph{i.e.}, $\delta$ is nondegenerate on $\mathcal{L}(f(X))$).
\item[(ii)] If $\dim X<\dim M-1$, then there is a residual subset $\mathcal{U}$ in the space of $C^{\infty}$ closed embeddings $f\co X\to M$ (with the strong $C^{\infty}$ topology) such that for every $f\in\mathcal{U}$ the Chekanov--Hofer pseudometric $\delta$ on $\mathcal{L}(f(X))$ vanishes identically.  If $X$ is compact, there is an integer $a$ depending only on $\dim M$ and $\dim X$ such that $\mathcal{U}$ is open in the $C^a$ topology. 
\end{itemize}
\end{theorem}

\begin{proof}  Part (i) is Theorem \ref{hypthm} and part (ii) is Corollary \ref{wtmain} (see Corollary \ref{wtmain} and what precedes it for the required value of $a$; we just mention here that one can take $a$ to be as small as $2$ provided that $\dim X<\binom{\dim M-\dim X+1}{2}$).  
\end{proof}

In what follows, a submanifold $N\subset M$ such that $\delta$ vanishes identically on $\mathcal{L}(N)$ will be called \textbf{weightless}; thus when $N$ is weightless, given any position to which $N$ can be moved by a Hamiltonian diffeomorphism, such a movement can be carried out in a way that requires arbitrarily little energy.  

Thus closed hypersurfaces (by Theorem \ref{int1}(i)) and compact Lagrangian submanifolds (by \cite{Ch00}) of geometrically bounded symplectic manifolds are CH-rigid, whereas by Theorem \ref{int1}(ii) there are many submanifolds having the opposite extreme property of being weightless.  Hypersurfaces and Lagrangian submanifolds of symplectic manifolds $(M,\omega)$ have a natural geometric property in common: they are \emph{coisotropic}.  (Recall that, where for a subspace $W$ of a symplectic vector space $(V,\omega)$ we denote by $W^{\omega}$ the $\omega$-orthogonal complement of $W$ in $V$, a submanifold $N\subset (M,\omega)$ is coisotropic provided that for all $x\in N$ we have $T_{x}N^{\omega}\subset T_x N$.)  One of the main themes of this paper is that the behavior of the Chekanov--Hofer pseudometric $\delta$ on $\mathcal{L}(N)$ is intimately related to how close $N$ is to being coisotropic.  Indeed we gather evidence for the following:

\begin{conj}\label{conjmain} Let $N$ be a compact submanifold of a geometrically bounded symplectic manifold $(M,\omega)$.  Then $N$ is CH-rigid if and only if $N$ is coisotropic.
\end{conj} 

We have: 

\begin{prop}\label{onlyif} The ``only if'' part of Conjecture \ref{conjmain} is true: indeed for any submanifold $N$ of any symplectic manifold $(M,\omega)$ such that $N$ is closed as a subset and $N$ is not coisotropic, $N$ is not CH-rigid. 
\end{prop}
\begin{proof}See Corollary \ref{ncdeg}.
\end{proof}

The ``if'' part of Conjecture \ref{conjmain} is consistent with the expectation, articulated for instance in \cite{Gi}, that coisotropic submanifolds should satisfy similar rigidity properties to Lagrangian submanifolds.\footnote{The hypothesis that $(M,\omega)$ is geometrically bounded  is necessary even in the Lagrangian case, as is shown by an example due to Sikorav described in \cite[Section 4]{Ch00}.}    As with other manifestations of this principle, its proof is obstructed by the lack of a suitable analogue of Lagrangian Floer theory for coisotropic submanifolds, but its proof becomes feasible if one imposes some additional hypotheses on the submanifold.  The following theorem illustrates this; we refer to Section \ref{coisosect} both for definitions and for other examples of hypotheses (in some cases rather more general, albeit more complicated) that are sufficient to guarantee CH-rigidity:

\begin{theorem} \label{coisomain} Let $N$ be a compact coisotropic submanifold of the symplectic manifold $(M,\omega)$, and assume that either \begin{itemize} \item[(i)] $M$ is geometrically bounded and $N$ is regular; or \item[(ii)] $M$ is compact, $N$ is stable, the group $\left\{\left.\int_{S^2}u^*\omega\right|u\co S^2\to N\right\}$ is discrete, and every leaf of the characteristic foliation of $N$ is dense in $N$.\end{itemize}
Then $N$ is CH-rigid.
\end{theorem}
\begin{proof} The first case is covered by Theorem \ref{regrigid}, and the second by Corollary \ref{posdisp}.
\end{proof}

\begin{remark}  For any $n\in \mathbb{Z}_+$ and any integer $k$ with $1\leq k\leq n$, Example \ref{ellipsoid} provides  compact coisotropic submanifolds $N_{k,n}$ of $\R^{2n}$ with codimension $k$ (namely, products of boundaries of ellipsoids) such that $N_{k,n}$ is CH-rigid.  By working in Darboux charts (and using the fact that Corollary \ref{denserigid} applies to arbitrary geometrically bounded ambient manifolds), one can replace $\R^{2n}$ by any $2n$-dimensional geometrically bounded symplectic manifold $(M,\omega)$.  On the other hand Theorem \ref{int1}(ii) shows that if $k\geq 2$ there are arbitrarily small smooth perturbations of $N_{k,n}$ in $M$ which are weightless.  Thus weightless and CH-rigid submanifolds coexist in all geometrically bounded symplectic manifolds $(M,\omega)$ and in all codimensions $k$ except those in which such coexistence is forbidden by  Theorem \ref{int1}(i) (which implies that a weightless submanifold has codimension $k\geq 2$) or Proposition \ref{onlyif} (which implies that a CH-rigid submanifold has codimension $k\leq n$). At the same time, for $2\leq k\leq n-1$ there are codimension $k$ submanifolds $N$ which are neither CH-rigid nor weightless: by Proposition \ref{onlyif}, Lemma \ref{RNlemma}(iii), and  Corollary \ref{lagopen} we could take $N$ to be any closed submanifold which contains a compact Lagrangian submanifold but is not coisotropic.\end{remark}

We now consider the opposite behavior, where $N$ is weightless, \emph{i.e.}, $\delta$ vanishes identically on the orbit $\mathcal{L}(N)$ of $N$ under the Hamiltonian diffeomorphism group.  Weightlessness is closely related to the lack of coisotropy of $N$.  Indeed, where a submanifold $N$ of $(M,\omega)$ is called \emph{nowhere coisotropic} if for all $x\in N$ it holds that $T_{x}N^{\omega}\not\subset T_x N$, we have:

\begin{theorem}\label{addint} All closed nowhere coisotropic submanifolds $N$ of a symplectic manifold $(M,\omega)$ are weightless.  
\end{theorem}
\begin{proof} See Corollary \ref{vanish}. 
\end{proof}

Of course any submanifold $N$ of $M$ such that $\dim N<\frac{1}{2}\dim M$ is nowhere coisotropic, as is any symplectic submanifold $N$ of $M$ of any positive codimension---thus such submanifolds are always weightless (when they are closed as subsets).  When $\dim X<\binom{\dim M-\dim X}{2}$, the residual set $\mathcal{U}$ of Theorem \ref{int1}(ii) (as constructed in Section \ref{genwt}) in fact consists precisely of nowhere coisotropic embeddings.  However once $\dim X\geq \binom{\dim M-\dim X}{2}$, it can no longer be expected to hold that nowhere coisotropic embeddings are dense in the space of closed embeddings, and $\mathcal{U}$ is taken to consist of embeddings $f\co X\to M$ which behave in a suitably generic way along their ``coisotropic loci'' $\{x\in X|(f_*T_{x}X)^{\omega}\subset f_*T_{x}X\}$; generally the relevant condition involves higher-order derivatives of $f$.

Various other sorts of rigidity or nonrigidity of subsets are often studied in symplectic topology; let us discuss the relation of CH-rigidity and weightlessness to some of these other notions.  First of all, recall that a closed subset $N\subset M$ is called displaceable if there is $\phi\in Ham(M,\omega)$ such that $\phi(N)\cap N=\varnothing$. While nondisplaceability is a sort of rigidity,  a CH-rigid subset can certainly be displaceable; for instance this holds if $M=\R^{2n}$ and $N$ is a compact hypersurface or a compact Lagrangian submanifold.  On the other hand a weightless submanifold $N$ might be nondisplaceable for trivial topological reasons, \emph{e.g.} if $[N]\cap [N]$ is nonzero in $H_{\dim M-2\dim N}(M)$.  Thus the behavior of the Chekanov--Hofer pseudometric $\delta$ is somewhat orthogonal to questions of displaceability.

The \emph{displacement energy} of  a closed subset $N\subset M$ is by definition \[ e(N,M)=\inf\left\{\|\phi\| \left| \phi\in Ham(M,\omega),\,\phi(N)\cap N=\varnothing\right.\right\};\]  $N$ might also be considered to be rigid if one has $e(N,M)>0$, and this notion of rigidity is somewhat more closely connected to ours.  Chekanov's original proof in \cite{Ch00} that compact Lagrangian submanifolds of geometrically bounded symplectic manifolds are CH-rigid used his famous theorem from \cite{Ch98} that such submanifolds always have positive displacement energy.   Indeed, rephrasing his argument into our language, he first showed that a Lagrangian submanifold would have to either be CH-rigid or weightless, and then he appealed to the following obvious fact to derive a contradiction:
\begin{prop} \label{obvious} If $N\subset M$ is a displaceable closed subset which is weightless then $e(N,M)=0$.
\end{prop}
\begin{proof} That $N$ is displaceable means that there is $N'\in\mathcal{L}(N)$ such that $N\cap N'=\varnothing$, and that $N$ is weightless implies that we have $\delta(N,N')=0$, \emph{i.e.} that for all $n\in \mathbb{Z}_+$ there is $\phi_n\in Ham(M,\omega)$ such that $\phi_n(N)=N'$ and $\|\phi_n\|<\frac{1}{n}$.  Since the $\phi_n$ disjoin $N$ from itself the result follows.
\end{proof}

It is not clear whether, conversely, if $e(N,M)=0$ then $N$ must be weightless.  It is true, though, that if there is a fixed $N'\in\mathcal{L}(N)$ such that $\delta(N,N')=0$ and $N\cap N'=\varnothing$ then $N$ is weightless, as may be deduced from Lemma \ref{RNlemma}.  In other words, rewriting the definition of $e(N,M)$ as \[ e(N,M)=\inf\left\{\delta(N,N')|N'\in\mathcal{L}(N),\,N\cap N'=\varnothing\right\},\] we see that if the above infimum  both is equal to zero \emph{and is attained} then $N$ is weightless.

Still another notion of nonrigidity for a (say compact) submanifold $N$ of a symplectic manifold $(M,\omega)$ is \emph{infinitesimal displaceability}: $N$ is said to be infinitesimally displaceable if there is a smooth function $H\co M\to \R$ whose Hamiltonian vector field $X_H$ has the property that, for all $x\in N$, $X_{H}(x)\notin T_xN$.  Since we assume that $N$ is compact, where $\{\phi_t\}_{t\in \R}$ denotes the flow of $X_H$, if $N$ is infinitesimally displaceable then we will have $\phi_t(N)\cap N=\varnothing$ for all sufficiently small nonzero $t$, in view of which $N$ clearly has $e(N,M)=0$.  We show in Proposition \ref{infwt} that in fact infinitesimally displaceable submanifolds are weightless.  However the converse need not be true, even if we assume that the weightless submanifold is displaceable and hence has zero displacement energy by Proposition \ref{obvious}; indeed there may be purely differential-topological obstructions
to the existence of the vector field $X_H$.  This leads to the following:

\begin{theorem}\label{inf0} Let $(M,\omega)$ be any $4k$-dimensional symplectic manifold where $k$ is a positive integer.  Then there is a compact submanifold $N\subset M$ of dimension $2k$ such that $N$ is not infinitesimally displaceable but $e(N,M)=0$.
\end{theorem}

\begin{proof} This will follow quickly from:
\begin{lemma} \label{nonormal} For any $k\in\mathbb{Z}_+$ there is a compact submanifold $N_0\subset \R^{4k}$ of dimension $2k$ such that the normal bundle of $N_0$ has no nonvanishing sections. 
\end{lemma}
To deduce Theorem \ref{inf0} from Lemma \ref{nonormal}, note that after composing the embedding of $N_0$ first with a suitable rescaling of $\R^{4k}$ and then with a Darboux chart for $(M,\omega)$, we can arrange for (a copy of) $N_0$ to be contained in the interior of a  closed Darboux ball $B$ which is displaceable in $M$.  By Theorem \ref{int1}(ii), arbitrarily $C^{\infty}$-close to this copy of $N_0$ there is a weightless submanifold $N$; in particular we can arrange for $N$ to still be contained in $B$ and to have normal bundle which is isomorphic to the normal bundle of $N_0$.  Since $N$ is weightless and, being contained in $B$, is displaceable, we have $e(N,M)=0$ by Proposition \ref{obvious}.  But $N$ cannot be infinitesimally displaceable, since any vector field which is nowhere tangent to $N$ would give rise to a nonvanishing section of the normal bundle to $N$.

Lemma \ref{nonormal} was originally proven by Mahowald in 1964 \cite{Ma64}, but here is a construction that symplectic topologists may find more appealing.  Let $Q^{2k}$ be the mapping torus of a reflection of the sphere $S^{2k-1}$.  Then where $T^{2k}$ is the $2k$-dimensional torus, using Lagrangian surgery as in \cite[Theorem 1a]{P91} one obtains a Lagrangian submanifold $N_0\subset \R^{4k}$ diffeomorphic to the connected sum $T^{2k}\#T^{2k}\# Q^{2k}$.  Now the Euler characteristic of $N_0$ is $-4$, and so the tangent bundle $TN_0$ has no nonvanishing sections (the nonorientability of $N_0$ is no problem here, see \emph{e.g.} \cite[Corollary 39.8]{St}).  But since $N_0$ is Lagrangian, its normal bundle is isomorphic to its tangent bundle.
\end{proof}

As far as I know, these are the first examples in the literature of submanifolds that are not infinitesimally displaceable but have zero displacement energy. It was essential for Lemma \ref{nonormal} that the submanifold $N_0$ was nonorientable, since if $N_0$ were orientable then the Euler class of the normal bundle of $N_0$ would be the restriction of a cohomology class from $\R^{4k}$ and so would be zero, and since $\dim N_0=\frac{1}{2}\dim \R^{4k}$ the Euler class of the normal bundle is the only obstruction to the existence of a nonvanishing section.  (In the nonorientable case the mod 2 Euler class necessarily vanishes for similar reasons, but the integral twisted  Euler class in the cohomology with local coefficients associated to the first Stiefel-Whitney class of the normal bundle can be nonvanishing, and it is this twisted Euler class which is the obstruction to finding a section.)  If we instead consider submanifolds $N_0\subset \R^{2n}$ of codimension less than $n$, then there will be higher-order obstructions to the existence of a nonvanishing normal vector field which can in principle be nontrivial when $N_0$ is orientable, though examples of this in the literature seem to be rather scarce.  Some examples of embeddings of orientable manifolds into Euclidean spaces for which the secondary obstruction to the existence of a normal section is nontrivial are given in \cite{Mas}, though in these cases the ambient Euclidean dimension is odd.  It seems likely that a product $N_0$ of two of Massey's examples would again admit no nonvanishing normal fields, and then the same argument as is used in the proof of Theorem \ref{inf0} would show that a small perturbation of $N_0$ has zero displacement energy without being infinitesimally displaceable.

\subsection{Outline of the paper} \label{outline} The upcoming Section \ref{norms} introduces some terminology and makes some observations concerning the pseudometrics that are induced on the orbits of the action of a group when that group is endowed with an invariant norm; of course the case of interest to us is the Hofer norm on the Hamiltonian diffeomorphism group of a symplectic manifold $(M,\omega)$, acting on closed subsets of $N$.  In particular we prove the simple but conceptually important Proposition \ref{stabclosure}, which connects the behavior of the pseudometric $\delta$ on the orbit $\mathcal{L}(N)$ of $N$ to the properties of the closure $\bar{\Sigma}_N$ of the stabilizer $\Sigma_N$ of 
$N$ with respect to the norm.

Section \ref{hyp} contains the proof of Theorem \ref{int1}(i), asserting that closed hypersurfaces are CH-rigid. Given existing results in the literature, this is much the easier half of that theorem: when the hypersurface separates the ambient manifold the result follows from the energy-capacity inequality proven in \cite[Theorem 1.1(ii)]{LM}, and the general case can be reduced to the (possibly disconnected) separating case by passing to finite covers.

Section \ref{rns} introduces a fundamental tool for the other main results of the paper: the \textbf{rigid locus} $R_N$ of a closed subset $N$ of a symplectic manifold $(M,\omega)$.  Where as before $\Sigma_N$ is the stabilizer of $N$ and $\bar{\Sigma}_N$ is its closure with respect to the Hofer norm, we have by definition \[ R_N=\bigcap_{\phi\in \bar{\Sigma}_N}\phi^{-1}(N).\]  Thus $R_N$ is a closed subset of $N$, invariant under the action of $\Sigma_N$ on $N$.  Among the key properties of $R_N$ are that $R_N=N$ if and only if $N$ is CH-rigid, while (modulo a trivial exception) $R_N=\varnothing$ if and only if $N$ is weightless.  The first of these statements is an obvious consequence of Proposition \ref{stabclosure}, but the second is deeper: its proof depends on Banyaga's fragmentation lemma.  Lemma \ref{rnp} then provides our main tool for proving either the failure of CH-rigidity or the weightlessness of a submanifold: by means of an explicit construction of certain kinds of elements of $\bar{\Sigma}_N$, we show that points at which a submanifold $N$ is not coisotropic cannot belong to $R_N$, from which Proposition \ref{onlyif} and Theorem \ref{addint} immediately follow; moreover Lemma \ref{rnp} is structured so as to facilitate an inductive argument which is later used in Section \ref{genwt} to prove Theorem \ref{int1}(ii) (asserting that generic closed embeddings of codimension at least two are weightless).  In the opposite direction we prove in Corollary \ref{lagopen} that (assuming $(M,\omega)$ to be geometrically bounded) if $N\subset M$ is a closed subset which contains a compact Lagrangian submanifold $L$, then we have $L\subset R_N$;  this gives a new proof of Chekanov's theorem from \cite{Ch00} that compact Lagrangian submanifolds are CH-rigid (by setting $L=N$), and it is later used to prove that some other classes of coisotropic submanifolds are CH-rigid as well.  Finally we prove Theorem \ref{stable}, which asserts that if a compact subset $N\subset M$ has a rigid locus with zero displacement energy, then the rigid locus of its ``stabilization'' $\hat{N}=N\times S^1\subset M\times \R^2$ is empty; this is used for some of the results of Section \ref{coisosect}.

Section \ref{coisosect} contains our results on the CH-rigidity of certain classes of coisotropic submanifolds, together with some illustrative examples.  These results fit roughly speaking into two rather distinct classes: those where the submanifold is CH-rigid because most of its points lie on compact Lagrangian submanifolds so that we can apply Corollary \ref{lagopen} (as we explain, this commonly occurs in the theory of symplectic reduction), and those where the submanifold is CH-rigid because it is stable in the sense of \cite{Gi} and because most of its points lie on dense leaves of the characteristic foliation, allowing us to make use of results from \cite{Gi} and \cite{U}.

Finally, Section \ref{genwt} proves Theorem \ref{int1}(ii).  The argument is rather involved, but here is a brief description of the idea.  Following the strategy introduced in Section \ref{rns}, the goal is to show that a generic submanifold $N\subset M$ of codimension larger than $1$ has empty rigid locus $R_N$.  Now a simple case of Lemma \ref{rnp} shows that $R_N\subset \{x\in N|T_{x}N^{\omega}\subset T_xN\}$.  Denoting the set on the right by $N_1$, one can use Thom's  jet transversality theorem to show that, for generic $N$, $N_1$ is a submanifold of $N$, with positive codimension since we assume that the codimension of $N$ is at least $2$.  Once we know that $R_N\subset N_1$ and that $N_1$ is a submanifold, another application of Lemma \ref{rnp} shows that in fact \[ R_N\subset \{x\in N|T_{x}N^{\omega}\subset T_xN_1\},\] and one can reasonably expect that the set on the right hand side above would generically be smaller than $N_1$.  This suggests an inductive scheme in which we produce, for  any positive integer $r$ and generic $N$ (with the precise genericity condition depending on $r$), a sequence of submanifolds $N=N_0\supset N_1\supset N_2\supset\cdots\supset N_r$, where each inclusion has positive codimension and repeated applications of Lemma \ref{rnp} show that if $R_N\subset N_i$ then $R_{N}\subset N_{i+1}$.  Since the dimensions of the $N_i$ are strictly decreasing the $N_i$ would eventually terminate in the empty set, implying that $R_N=\varnothing$ and hence that $N$ is weightless.  This is essentially what we do, modulo a technical issue that forces us to work separately in each member of a countable (finite if $N$ is compact) open cover of $N$.  The statement that $N_i$ is a manifold of the expected dimension is obtained by appealing to the jet transversality theorem for the $i$-jet of the embedding of $N$.

\subsection{Notation and Conventions}\label{conv}\begin{itemize}\item  All manifolds and submanifolds are assumed to be without boundary unless the modifier ``with boundary'' is explicitly added.
\item Submanifolds are always assumed to be embedded.  A ``closed submanifold'' $N$ of a manifold $M$ is a submanifold of $M$ which is closed as a subset (it need not be compact if $M$ is not compact).
\item If $(M,\omega)$ is a symplectic manifold, a compactly supported smooth function $H\co [0,1]\times M\to\R$ determines a time-dependent Hamiltonian vector field $\{X_{H_t}\}_{0\leq t\leq 1}$ by the prescription that $\omega(\cdot,X_{H_t})=d\left(H(t,\cdot)\right)$.
\item $Ham(M,\omega)$ is the group of time-one maps of the time-dependent Hamiltonian vector fields generated by \emph{compactly-supported} 
smooth functions $H\co [0,1]\times M\to \R$.  
\item If $V\subset M$ is an open subset, then $Ham^c(V)$ denotes the subgroup of $Ham(M,\omega)$ consisting of those time-one maps of time-dependent Hamiltonian vector fields generated by Hamiltonian functions $H\co [0,1]\times M\to \R$ having compact support contained in $[0,1]\times V$.
\item For a closed subset $N\subset M$, we denote by $\Sigma_N$ the subgroup of $Ham(M,\omega)$ consisting of Hamiltonian diffeomorphisms $\phi$ such that $\phi(N)=N$.
\end{itemize}

\section{Norms on groups}\label{norms}

We collect in this section some conventions and observations regarding norms on groups and their associated homogeneous spaces; this will serve as part of the framework for the rest of the paper.

\begin{dfn} If $G$ is a group, an \emph{invariant norm} on $G$ is a map $\|\cdot\|\co G\to [0,\infty)$ with the following properties:
\begin{itemize}\item For $g\in G$, $\|g\|\geq 0$ with equality if and only if $g$ is the identity.
\item For all $g,h\in G$ we have: \begin{align*} \|g^{-1}\|&=\|g\| \\ \|gh\|&\leq \|g\|+\|h\| \\ \|h^{-1}gh\|&=\|g\| \end{align*}
\end{itemize}
\end{dfn}

Invariant norms on a group $G$ are in one-to-one correspondence with bi-invariant metrics: given an invariant norm $\|\cdot\|$ one obtains a bi-invariant metric $d$ by setting $d(g,h)=\|gh^{-1}\|$, and conversely one can recover $\|\cdot\|$ from $d$ by setting $\|g\|=d(g,e)$ where $e$ is the identity.  In particular an invariant norm on $G$ induces naturally a (metric) topology on $G$, with respect to which $G$ is readily seen to be a topological group. 

Now suppose that $G$ acts transitively on the left on some set $S$.  Associated to the invariant norm $\|\cdot\|$ is a function $\delta\co S\times S\to [0,\infty)$ defined by \[ \delta(s_1,s_2)=\inf\left\{\|g\|\left|gs_1=s_2\right.\right\}.\]

It is straightforward to verify from the axioms for $\|\cdot\|$ that $\delta$ defines a $G$-invariant pseudometric on $S$: in other words we have, for $s_1,s_2,s_3\in S$ and $g\in S$, the following identities: \begin{align*} \delta(s_1,s_1)&=0 \\ \delta(s_1,s_2)&=\delta(s_2,s_1) \\ \delta(s_1,s_3)&\leq \delta(s_1,s_2)+\delta(s_2,s_3) \\ \delta(gs_1,gs_2)&=\delta(s_1,s_2).\end{align*}

Whether the pseudometric $\delta$ on the $G$-set $S$ is in fact a metric (\emph{i.e.}, whether it holds that $\delta(s_1,s_2)>0$ whenever $s_1\neq s_2$) is a more subtle issue, which is partly addressed by the following:

\begin{prop}\label{stabclosure} Let $\|\cdot\|$ be an invariant norm on the group $G$, which acts transitively on the left on the set $S$, inducing the invariant pseudometric $\delta$ as above.  Choose a basepoint $s_0\in S$, and define \[ H=\{g\in G|gs_0=s_0\}.\]  Then the closure of $H$ with respect to the topology on $G$ induced by $\|\cdot\|$ is a subgroup, and is given by \begin{equation}\label{Hbar} \bar{H}=\{g\in G|\delta(s_0,gs_0)=0\}.\end{equation}  In particular, $\delta$ is a metric on $S$ if and only if $H$ is closed.
\end{prop}

\begin{proof} The fact that $\bar{H}$ is a subgroup just follows from the general elementary fact that, in any topological group, the closure of a subgroup is still a subgroup.

 If $g\in \bar{H}$, then for all $\ep>0$ there is $h\in H$ such that $\|gh^{-1}\|<\ep$.  Since $h\in H$ we have $h^{-1}s_0=s_0$.  Thus $\delta(s_0,gs_0)=\delta(s_0,gh^{-1}s_0)<\ep$.  Since $\ep>0$ was arbitrary this shows that $\delta(s_0,gs_0)=0$.

Conversely, if $\delta(s_0,gs_0)=0$, by the definition of $\delta$ for any $\ep>0$ we can find $h\in G$ such that $hs_0=gs_0$ and $\|h\|<\ep$.  Then $h^{-1}gs_0=s_0$, \emph{i.e.} $h^{-1}g\in H$, and we have $d(h^{-1}g,g)=\|h^{-1}gg^{-1}\|=\|h\|<\ep$.  Since $\ep$ was arbitrary this shows that $g\in \bar{H}$.

This proves the characterization (\ref{Hbar}) of $\bar{H}$. The last sentence follows immediately: if $H=\bar{H}$ then the required nondegeneracy  holds using the $G$-invariance of $\delta$ and the transitivity of the action, while if $\bar{H}\setminus H$ contains some element $g$ then we will have $\delta(s_0,gs_0)=0$ even though $s_0\neq gs_0$.
\end{proof}

The rest of the paper specializes to the following situation.  
Let $(M,\omega)$ be a symplectic manifold and let $N\subset M$ be a  closed subset.  For the group $G$ we use the group $Ham(M,\omega)$ of compactly-supported Hamiltonian diffeomorphisms of $M$, and for the set $S$ we use \[ S=\mathcal{L}(N):=\{\phi(N)|\phi\in Ham(M,\omega)\}.\]  On $Ham(M,\omega)$ we have the Hofer norm of \cite{Ho} (which was proven to be nondegenerate on all symplectic manifolds on \cite{LM}): where for a smooth compactly supported function $H\co [0,1]\times M\to \R$ we denote by $\phi_{H}^{1}$ the time-one map of $H$, one sets, for $\phi\in Ham(M,\omega)$, \[ \|\phi\|=\inf\left\{\left.\int_{0}^{1}\left(\max_M H(t,\cdot)-\min_M H(t,\cdot)\right)dt\right|\phi_{H}^{1}=\phi\right\}.\]  

Using the obvious left action of $Ham(M,\omega)$ on $\mathcal{L}(N)$, the Hofer norm $\|\cdot\|$ induces a ``Chekanov--Hofer'' pseudometric $\delta$ on $\mathcal{L}(N)$.  
As in the introduction, we use the following shorthand:

\begin{dfn}\begin{itemize}\item[(i)] A closed subset $N\subset M$ is called \emph{weightless} if the Chekanov--Hofer pseudometric $\delta$ on $\mathcal{L}(N)$ vanishes identically.
\item[(ii)] A closed subset $N\subset N$ is called \emph{CH-rigid} if the Chekanov--Hofer pseudometric $\delta$ on $\mathcal{L}(N)$ is a nondegenerate metric.\end{itemize}
\end{dfn}

In other words, $N$ is weightless if, whenever $N'\subset M$ has the property that $N'=\phi(N)$ for some $\phi\in Ham(M,\omega)$, the diffeomorphism $\phi$ can in fact be chosen to have arbitrarily low energy, while $N$ is CH-rigid if this holds only when $N=N'$.

Where as in Section \ref{conv} $\Sigma_N$ denotes the stabilizer of $N$ under the action of $Ham(M,\omega)$ (\emph{i.e.}, $\Sigma_N=\{\phi\in Ham(M,\omega)|\phi(N)=N\}$), Proposition \ref{stabclosure} provides another characterization of these properties: $N$ is CH-rigid if and only if $\bar{\Sigma}_N=\Sigma_N$, while $N$ is weightless if and only if $\bar{\Sigma}_N=Ham(M,\omega)$, where of course $\bar{\Sigma}_N$ denotes the closure of $\Sigma_N$ in $Ham(M,\omega)$ with respect to the Hofer norm.

\section{Hypersurfaces}\label{hyp}

The goal of this section is to prove Theorem 1.1(i), asserting that closed connected hypersurfaces in symplectic manifolds are CH-rigid.  As we will see, this follows fairly quickly from the energy-capacity inequality together with covering tricks.

\begin{lemma}\label{sep} Where $(M,\omega)$ is a connected symplectic manifold, let $N\subset M$ be a (not necessarily connected) closed subset with the property that $M\setminus N=M_0\cup M_1$ where $M_0$ and $M_1$ are disjoint nonempty connected open subsets of $M$ and $\bar{M}_i=M_i\cup N$ for $i=0,1$.  Then $N$ is CH-rigid.
\end{lemma}

\begin{proof}  Let $N'\in\mathcal{L}(N)\setminus \{N\}$; we are to show that there is $\delta>0$ such that any $\psi\in Ham(M,\omega)$ with $\psi(N)=N'$ has $\|\psi\|\geq\delta$.  Of course the assumption that $N'\in\mathcal{L}(N)$ means that there is some $\psi_0\in Ham(M,\omega)$ with $N'=\psi_0(N)$.  If we set $M'_i=\psi_0(M_i)$ for $i=0,1$, it will then hold that $M\setminus N'=M'_0\cup M'_1$ where the $M'_i$ are disjoint nonempty connected open sets with $\bar{M}'_i=M'_i\cup N'$.  

Since $N\neq N'$, one (more likely both) of $N\setminus N'$ and $N'\setminus N$ is nonempty.  Suppose the former set is nonempty, and choose $x_0\in N\setminus N'$.  Since $x_0$ lies in the closures of both $M_0$ and $M_1$, any open set containing $x_0$ will intersect both $M_0$ and $M_1$.  On the other hand since $x_0\notin N'$, for some $j\in\{0,1\}$ we have $x_0\in  M'_j$.  Thus in particular $M'_j$ intersects both $M_0$ and $M_1$.  We claim that $\delta(N,N')$ is at least equal to the minimum of the displacement energies of $M_0\cap M'_j$ and $M_1\cap M'_j$, which of course is positive by \cite[Theorem 1.1(ii)]{LM} since $M_0\cap M'_j$ and $M_1\cap M'_j$ are nonempty open sets.  Indeed, if $\psi\in Ham(M,\omega)$ has $\psi(N)=N'$, then also $\psi(M\setminus N)=M\setminus N'$, so since the connected components of $M\setminus N$ are $M_0$ and $M_1$ while those of $M\setminus N'$ are $M'_0$ and $M'_1$, it holds that either $\psi(M_0)\cap M'_j=\varnothing$ or $\psi(M_1)\cap M'_j=\varnothing$.  In the first case $\psi$ displaces $M_0\cap M'_j$, and in the second case it displaces $M_1\cap M'_j$, proving that in either case $\|\psi\|$ is at least the minimum of the two aforementioned displacement energies.

This proves the result in the case that $N\setminus N'\neq\varnothing$.  The case that $N'\setminus N\neq \varnothing$ is essentially identical: one will have that $M_j$ intersects both $M'_0$ and $M'_1$ for some $j$, and then (using that $\|\psi\|=\|\psi^{-1}\|$) one proves in the same way as in the previous paragraph that $\delta(N,N')$ is at least the minimum of the displacement energies of  $M'_0\cap M_j$ and  $M'_1\cap M_j$.
\end{proof}

The proof that all closed codimension-one submanifolds, and not just separating ones, are CH-rigid proceeds by passing to finite covers in order to appeal to Lemma \ref{sep}.  The following simple lemma is the basis for this:

\begin{lemma}\label{covlem}
Let $\pi\co X\to M$ be a (surjective) finite covering map where $(M,\omega)$ is a symplectic manifold.  Suppose that $N\subset M$ is a closed subset such that $\pi^{-1}(N)$ is CH-rigid as a subset of $(X,\pi^*\omega)$.  Then $N$ is CH-rigid as a subset of $(M,\omega)$.
\end{lemma}

\begin{proof} If $H\co [0,1]\times M\to \R$ is any compactly supported smooth function, then the function $\tilde{H}(t,x)=H(t,\pi(x))$ on $[0,1]\times X$ will still be compactly supported since $\pi$ is a finite covering map, and the Hamiltonian flow generated by $\tilde{H}$ will lift the flow generated by $H$.  This gives rise to a map \begin{align*} \widetilde{Ham}(M,\omega)&\to \widetilde{Ham}(X,\omega_X)\\ \phi&\mapsto \tilde{\phi}\end{align*} between the universal covers of the respective Hamiltonian diffeomorphism groups. Continue to denote by $\|\cdot\|$ the Hofer (pseudo-)norm on $\widetilde{Ham}$ obtained by taking infima of lengths of paths in a given homotopy class, and note that the Chekanov--Hofer pseudometric is  given by the formula \[ \delta(N,N')=\inf\{\|\phi\| | \phi\in \widetilde{Ham}(M,\omega),\,\phi(N)=N'\}\] where we take the infimum over pseudonorms of elements of $\widetilde{Ham}$ rather than over norms of elements of $Ham$ and where for $\phi\in\widetilde{Ham}(M,\omega)$ we denote by $\phi(N)$ the image of $N$ under the terminal point of a path in $Ham(M,\omega)$ representing the homotopy class $\phi$.  Moreover
 we have $\|\tilde{\phi}\|\leq \|\phi\|$ for all $\phi\in \widetilde{Ham}(M,\omega)$, by virtue of the fact that $\max\tilde{H}(t,\cdot)-\min\tilde{H}(t,\cdot)=\max H(t,\cdot)-\min H(t,\cdot)$.  From this it follows that, for any $N'\in \mathcal{L}(N)$, we have \[ \delta(N,N')\geq \delta\left(\pi^{-1}(N),\pi^{-1}(N')\right).\]  If $N'\neq N$, then since $\pi$ is surjective $\pi^{-1}(N')\neq \pi^{-1}(N)$, so by the hypothesis of the lemma $\delta(\pi^{-1}(N),\pi^{-1}(N'))> 0$, whence $\delta(N,N')>0$, proving that $N$ is CH-rigid.
\end{proof}

\begin{lemma}\label{orrigid} Let $(M,\omega)$ be a connected symplectic manifold and $N\subset M$ a connected orientable codimension-one submanifold which is closed as a subset.  Then $N$ is CH-rigid.
\end{lemma}

\begin{proof} Since $N$ and $M$ are orientable, the normal bundle to $N$ in $M$ is orientable and hence trivial since it has rank one.  Thus by the tubular neighborhood theorem there is a neighborhood $U$ of $N$ in $M$ and a diffeomorphism $\Phi\co U\to \R\times N$ which restricts to $N$ as the map $n\mapsto (0,n)$.  Let \[ U_+=\Phi^{-1}\left((0,\infty)\times N\right)\quad\mbox{and}\quad U_-=\Phi^{-1}\left((-\infty,0)\times N\right) \] 
Since $M$ and $N$ are assumed connected it is easy to see that $M\setminus N$ has either one or two path components; the case where $M\setminus N$ has two path components is covered by Lemma \ref{sep}, so let us assume that $M\setminus N$ is connected.    Let $U^1$ and $U^2$ be two identical copies of $U$, containing open subsets $U^{1}_{\pm}$, $U^{2}_{\pm}$ as above.  Let $X_0$ denote the  manifold obtained from $(M\setminus N)\coprod U^1\coprod U^2$ by identifying points of $U_-\subset M\setminus N$ with those of $U^{1}_{-}$, and points of $U_+\subset M\setminus N$ with those of $U^{2}_{+}$.  (So $X_0$ is diffeomorphic to $M\setminus N$, but with the ends $U_{\pm}$ ``elongated'' to disjoint copies $U^1$ and $U^2$ of $U$).  Now let $X_{1}$ and $X_{2}$ be two identical copies of $X_0$, so that we have copies of $U$ as above embedded as $U^{1}_{1}$ and $U_{1}^{2}$ in $X_1$, and as $U_{2}^{1}$ and $U^{2}_{2}$ in $X_2$, and let $X$ be the  manifold obtained from $X_1\coprod X_2$ identifying $U^{1}_{1}$ with $U^{2}_{2}$, and $U^{2}_{1}$ with $U_{2}^{1}$.  (See Figure \ref{figure}.)

\begin{figure}\label{figure}
\centering 
\includegraphics[width=4in]{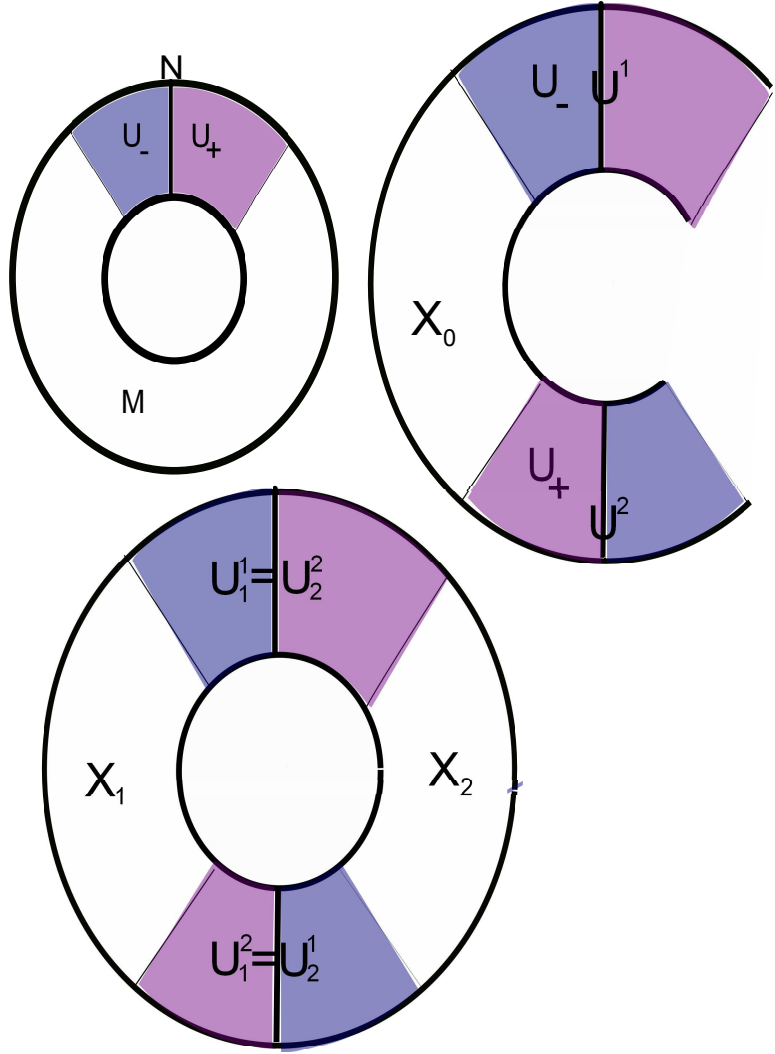}
\caption{The manifolds $M$, $X_0$, and $X=X_1\cup X_2$ in the proof of Lemma \ref{orrigid}.  We have an obvious double cover $\pi\co X\to M$, and $\pi^{-1}(N)\subset X$ (which appears in the figure as a union of two solid vertical line segments) separates $X$.}
\end{figure}

Every point of $X$ is a member of either (or both) a copy of $M\setminus N$ or a copy of $U$, and so we get a map $\pi\co X\to M$ obtained from the inclusions of $M\setminus N$ and $U$ into $M$.  It is easy to see that $\pi$ is a two-to-one covering map, such that $X\setminus \pi^{-1}(N)$ is a disjoint union of two copies of $M\setminus N$, each having boundary given by $\pi^{-1}(N)$.  Let $\omega_X=\pi^*\omega$. By Lemma \ref{sep}, $\pi^{-1}(N)\subset X$ is then CH-rigid, so by Lemma \ref{covlem} $N\subset M$ is also CH-rigid.
\end{proof}
 
So for the following theorem, which restates Theorem \ref{int1}(i), it remains only to address the nonorientable case, which can likewise be handled by a covering argument:

\begin{theorem}\label{hypthm} For any symplectic manifold $(M,\omega)$, any connected submanifold $N\subset M$ of codimension one which is closed as a subset is CH-rigid.
\end{theorem}

\begin{proof} Let $\nu\to N$ denote the normal bundle to $N$ in $M$.  Since $N$ is assumed closed as a subset, the inclusion of $N$ into $M$ is a proper map, so $N$ has a mod 2 Poincar\'e dual $PD(N)\in H^1(M;\mathbb{Z}/2)$, and $PD(N)|_N$ coincides with the mod 2 Euler class (\emph{i.e.}, the first Stiefel--Whitney class) $w_1(\nu)\in H^1(N;\mathbb{Z}/2)$.  Let $\pi\co X\to M$ be the cover associated to the kernel of the evaluation map $PD(N)\co \pi_1(M)\to\mathbb{Z}/2$, so $\pi$ is a two-to-one cover if $\nu$ is nonorientable and the identity otherwise, and in any case we have $\pi^*PD(N)=0$.  Then where $\tilde{N}=\pi^{-1}(N)$, the normal bundle $\tilde{\nu}$ of $\tilde{N}$ in $X$ is given by $\tilde{\nu}=\pi^*\nu$, and so we have \[ w_1(\tilde{\nu})=\pi^*\left(PD(N)|_N\right)=\left.\left(\pi^*PD(N)\right)\right|_{\tilde{N}}=0.\]  So the normal bundle to $\tilde{N}$ in $X$ is orientable, and so since $X$ is also orientable it follows that $\tilde{N}$ is orientable.  Of course if $N$ is orientable then $\tilde{N}=N$ and $X=M$, but if $N$ is not orientable then $\pi|_{\tilde{N}}\co \tilde{N}\to N$ is the orientable double cover of $N$ and in particular is connected.  So Lemma \ref{orrigid} applies to show that $\tilde{N}$ is CH-rigid, and so by Lemma \ref{covlem} $N$ is also CH-rigid.
\end{proof}

\section{The rigid locus}\label{rns}

This section proves basic properties concerning our most important tool in this paper, the rigid locus of a closed subset of a symplectic manifold.  Using Lemmas \ref{RNlemma} and \ref{rnp}, we will quickly prove Proposition \ref{onlyif} and Theorem \ref{addint}, and lay part of the foundation for the proof of Theorem \ref{int1}(ii), which will be proven later in Section \ref{genwt}. Also, Sections \ref{lagsec} and \ref{unstable} will prove properties of the rigid locus that will be important in the proof of Theorem \ref{coisomain} in Section \ref{coisosect}.

We consider general closed subsets $N$ of the symplectic manifold $(M,\omega)$.  As before, $\mathcal{L}(N)$ denotes the orbit of $N$ under $Ham(M,\omega)$, $\delta$ denotes the pseudometric on $\mathcal{L}(N)$ induced by the Hofer norm, $\Sigma_N$ denotes the stabilizer $\{\phi\in Ham(M,\omega)|\phi(N)=N\}$, and $\bar{\Sigma}_N$ is the closure of $\Sigma_N$ with respect to the Hofer norm.

\begin{dfn} If $N$ is a closed subset of $M$, the \emph{rigid locus} of $N$ is the subset \[ R_N=\{x\in N|(\forall \phi\in \bar{\Sigma}_N)(\phi(x)\in N)\}.\]
\end{dfn}

\begin{lemma} \label{RNlemma} If $N$ is a proper closed subset of the symplectic manifold $(M,\omega)$, the rigid locus $R_N\subset N$ obeys the following properties.
\begin{itemize} \item[(i)] $R_N$ is a closed subset of $N$.
\item[(ii)] $R_N=N$ if and only if $N$ is CH-rigid.
\item[(iii)] If $R_N=\varnothing$ then $N$ is weightless.  Conversely, assuming that no connected component of $M$ is contained in $N$, if $N$ is weightless then $R_N=\varnothing$.
\item[(iv)] For all $\psi\in \bar{\Sigma}_N$ we have $\psi(R_N)=R_N$.
\item[(v)] Suppose that $N'\in \mathcal{L}(N)$ has the property that $\delta(N,N')=0$.  Then $R_N\subset N\cap N'$.
\end{itemize}
\end{lemma}

\begin{proof} For (i), simply note that $R_N=\cap_{\phi\in \bar{\Sigma}_N}\phi^{-1}(N)$ and $N$ is assumed to be a closed subset of $M$. So $R_N$ is closed as a subset of $M$, hence also as a subset of $N$.

For (ii), if $\Sigma_N=\bar{\Sigma}_N$ then clearly $R_N=N$.  Conversely if there exists some $\phi\in \bar{\Sigma}_N\setminus \Sigma_N$, then either $\phi(N)\setminus N$ or $N\setminus \phi(N)$ is nonempty.  In the first case we find $x\in N$ with $\phi(x)\notin N$, so $x\notin R_N$, while in the second case we find $x\in N$ with $\phi^{-1}(x)\notin N$, and so since $\phi^{-1}\in \bar{\Sigma}_N$ again $x\notin R_N$.  So in any event if $\bar{\Sigma}_N\neq \Sigma_N$ then $N\neq R_N$. By Proposition \ref{stabclosure} this proves that the nondegeneracy of the pseudometric is equivalent to the condition that $R_N=N$.

For the second half of (iii), suppose that $N$ is weightless and that no connected component of $M$ is contained in $N$.   We then have $\bar{\Sigma}_N=Ham(M,\omega)$, and since $Ham(M,\omega)$ acts transitively on each of its connected components this implies that $R_N=\varnothing$, as any point in $N$ can be moved by an element of $\bar{\Sigma}_N$ to a point in the same connected component of $M$ which is not in $N$.  Now let us prove the first half of (iii) (which is perhaps the only nontrivial part of this lemma).   Suppose that $R_N=\varnothing$, so that for each $x\in N$ we can find $\phi_x\in \bar{\Sigma}_N$ so that $\phi_x(x)\notin N$.  We can then find an open-in-$M$ neighborhood of $x$, say $U_x$, so that $\phi_x(U_x)\cap N=\varnothing$.  We claim that this implies that $Ham^c(U_x)\leq \bar{\Sigma}_N$ (where $Ham^c(U_x)$ is the group of diffeomorphisms generated by Hamiltonians compactly supported in $[0,1]\times U_x$).  Indeed, if $\psi\in Ham^c(U_x)$ and $y\in N$, then $\phi_{x}^{-1}(y)\notin U_x$, and so $(\phi_{x}\circ\psi\circ\phi_{x}^{-1})(y)=y$.  Thus whenever $\psi\in Ham^c(U_x)$ we have $\phi_{x}\circ\psi\circ\phi_{x}^{-1}\in \Sigma_N$.  So since $\bar{\Sigma}_N$ is a subgroup of $Ham(M,\omega)$ which contains both $\phi_x$ and $\Sigma_N$ it follows that $Ham^c(U_x)\leq \bar{\Sigma}_N$.  
 Thus, if $R_N=\varnothing$, we have an open cover \[ M=(M\setminus N)\cup \bigcup_{x\in N}U_x,\] where $Ham^c(U_x)\leq \bar{\Sigma}_N$ by what we have just shown, and where $Ham^c(M\setminus N)\leq \bar{\Sigma}_N$ since all elements of $Ham^c(M\setminus N)$ act trivially on $N$.  But Banyaga's fragmentation lemma \cite[III.3.2]{Ban} asserts that all of $Ham(M,\omega)$ is generated by Hamiltonian diffeomorphisms supported within the members of any given open cover.   So since $\bar{\Sigma}_N$ is a subgroup of $Ham(M,\omega)$ it must in fact be equal to all of $Ham(M,\omega)$, which by Proposition \ref{stabclosure} implies that $\delta$ vanishes identically, \emph{i.e.} that $N$ is weightless.
 
(iv) is essentially immediate from the definition and the fact that $\bar{\Sigma}_N$ is a group: if $x\in R_N$ and $\psi\in \bar{\Sigma}_N$ then for all $\phi\in\bar{\Sigma}_N$ we will have $\phi\circ\psi\in\bar{\Sigma}_N$ and so $\phi(\psi(x))\in N$, proving that $\psi(R_N)\subset R_N$.  The reverse inclusion follows by the same argument applied to $\psi^{-1}$ rather than $\psi$.

For (v), by Proposition \ref{stabclosure} if $\delta(N,N')=0$ we can write $N'=\phi(N)$ where $\phi\in \bar{\Sigma}_N$.  If $x\in R_N$, then since $x\in N$ obviously we have $\phi(x)\in N'$, while also $\phi(x)\in N$ by the definition of $R_N$.  So $\phi(R_N)\subset N\cap N'$.  But by (iv)  we have $\phi(R_N)=R_N$.
\end{proof}

\begin{lemma}\label{rnp}  Assume that the closed subset $N\subset M$ is a submanifold, let $\mathcal{O}\subset N$ be an open subset, and suppose that for some relatively closed subset $P\subset \mathcal{O}$ which is also a submanifold we have $\mathcal{O}\cap R_N\subset P$.  Then $\mathcal{O}\cap R_N\subset \{x\in P|T_{x}N^{\omega}\subset T_x P\}$.
\end{lemma}

\begin{remark} This lemma may be slightly easier to decipher if one puts both $\mathcal{O}=N$ and $P=N$ (so that the condition $\mathcal{O}\cap R_N\subset P$ is vacuous)---in this case the conclusion is that $R_N$ is necessarily contained in the set of points $x$ at which $T_xN$ is a coisotropic subspace of $T_x M$.  Once one knows this, if this ``coisotropic locus'' is a smooth manifold, then one can apply the lemma again with $P$ equal to the coisotropic locus, and so conclude that $R_N$ is contained in a (possibly) still smaller set.  Indeed  this procedure can be iterated indefinitely; this is roughly speaking what we do in Section \ref{genwt}.
\end{remark}

\begin{proof}[Proof of Lemma \ref{rnp}]   Suppose that $x\in P$ does \emph{not} have the property that $T_xN^{\omega}\subset T_x P$; we will show that $x\notin R_N$.  

Taking $\omega$-orthogonal complements, our assumption on $x$ is equivalent to the statement  that there exists some element  $v\in T_{x}P^{\omega}\setminus T_{x}N$. We may then choose a smooth compactly-supported function $H\co M\to \R$  such that $H|_N=0$ but $dH(v)>0$.

For each positive integer $n$ let $g_n\co \R\to\R$ be a smooth function such that $g_n(s)=0$ for $|s|<\frac{1}{n}$,  $g_n(s)=s$ for $|s|>\frac{2}{n}$, and $g_{n}'(s)\geq 0$ for all $s$. 
Now define functions $H_n\co M\to \R$ by $H_n=g_n\circ H$.  Let $(\phi_{n}^{t})$ and $(\phi^t)$ denote the time-$t$ Hamiltonian flows of the functions $H_n$ and $H$ respectively.  

Now $H_n$ vanishes identically on a neighborhood of $N$ (namely $\{y ||H(y)|<1/n\}$), so $\phi_{n}^{t}$ acts as the identity on $N$ and so certainly $\phi_{n}^{t}\in \Sigma_{N}$ for all $n$ and $t$.  
Meanwhile since $g_n$ converges uniformly to the identity it holds that $H_n\to H$ uniformly, and so $\phi_{n}^{t}\to \phi_{n}$ with respect to the Hofer metric for all $t$.  Thus each $\phi^t\in \bar{\Sigma}_N$.

The function $H$ which generates the flow $(\phi^t)$ has $dH(v)>0$, where $v\in T_{x}P^{\omega}\setminus T_{x}N$.  Where $X_H$ is the Hamiltonian vector field of $H$, we thus have $\omega_x(v,X_H)\neq 0$, and so since $v\in T_{x}P^{\omega}$ we have $X_H(x)\notin T_xP$.  So for sufficiently small nonzero $t$ it will hold that $\phi^t(x)\notin P$ but $\phi^t(x)\in \mathcal{O}$.  But by Lemma \ref{RNlemma}(iv) we will have $\phi^t(R_N)=R_N$ for all $t$. So since $\mathcal{O}\cap R_N\subset P$ by assumption, it must be that $x\notin R_N$, as desired.
\end{proof}
 
\begin{cor}\label{ncdeg} Let $N\subset M$ be any submanifold which is not coisotropic.  Then the Chekanov--Hofer pseudometric $\delta$ on $\mathcal{L}(N)$ is degenerate (\emph{i.e.}, $N$ is not CH-rigid).
\end{cor}

\begin{proof} Applying Lemma \ref{rnp} with $\mathcal{O}=P=N$, we see that if $R_N=N$ then we must have $T_{x}N^{\omega}\subset T_{x}N$ for all $x\in N$, \emph{i.e.} $N$ is coisotropic.  So if $N$ is not coisotropic then $R_N\neq N$, so by Lemma \ref{RNlemma}(ii) $\delta$ must be degenerate.
\end{proof}

\begin{dfn} A submanifold $N$ of a symplectic manifold $(M,\omega)$ is called \emph{nowhere coisotropic} if for all $x\in N$ we have $T_xN^{\omega}\setminus T_x N\neq\varnothing$.
\end{dfn}

The following restates Theorem \ref{addint}.

\begin{cor} \label{vanish} Let $N$ be a submanifold of the symplectic manifold $(M,\omega)$ which is closed as a subset and is nowhere coisotropic.  Then $N$ is weightless.
\end{cor}

\begin{proof} Again applying Lemma \ref{rnp} with $P=N$, we see that if $N$ is nowhere coisotropic then we must have $R_N=\varnothing$, which implies that $N$ is weightless by Lemma \ref{RNlemma}(iii).
\end{proof}

Recall that a compact submanifold $N$ of a symplectic manifold $(M,\omega)$ is called \emph{infinitesimally displaceable} if there is a smooth function $H\co M\to \R$ such that the Hamiltonian vector field $X_H$ of $H$ has the property that $X_H(x)\notin T_xN$ for all $x\in N$.  Of course for $N$ to be infinitesimally displaceable it is necessary for the normal bundle of $N$ in $M$ to have a nowhere-vanishing section.  Conversely, results of \cite{LS}, \cite{P95}, and \cite{Gu} show that if $N$  is nowhere coisotropic, or if $\dim N=\frac{1}{2}\dim M$ but $N$ is  not Lagrangian, then $N$ will be infinitesimally displaceable provided that its normal bundle has a nowhere-vanishing section.  We have, somewhat consistently with Corollary \ref{vanish}:

\begin{prop}\label{infwt} If $N\subset M$ is a compact submanifold which is infinitesimally displaceable then $N$ is weightless.
\end{prop}

\begin{proof} Choose a compactly-supported Hamiltonian $H\co M\to \R$ so that $X_H$ is nowhere-tangent to $N$; by rescaling we may as well assume that $\max H-\min H=1$.  For any $t\in \R$ let $\phi_t$ denote the time-$t$ flow of $X_H$, so we have $\|\phi_t\|\leq |t|$ for all $t$.  Since $X_H$ is nowhere-tangent to $N$ and $N$ is compact, we may choose $\ep_0>0$ so that \[ \phi_t(N)\cap N=\varnothing\quad\mbox{ whenever }0<|t|\leq\ep_0.\] Let  $\eta$ be any number with $0<\eta<\ep_0$, and let $\beta\co M\to [0,1]$ be a smooth function such that $\beta=1$ on a neighborhood of $\cup_{t\in [\eta,\ep_0]}\phi_t(N)$ and $\beta=0$ on a neighborhood of $N$.  Let $K=\beta H$, and let $\{\psi_t\}$ be the Hamiltonian flow of $K$.  Then since $K$ vanishes on a neighborhood of $N$ we have $\psi_t(N)=N$ for all $t$.  Meanwhile since $K$ coincides with $H$ on a neighborhood of $\cup_{t\in [\eta,\ep_0]}\phi_t(N)$, and since $\phi_t(\phi_{\eta}(N))$ remains in this neighborhood for all $t\in [0,\ep_0-\eta]$, we have \[ \psi_{\ep_0-\eta}(\phi_{\eta}(N))=\phi_{\ep_0-\eta}(\phi_{\eta}(N))=\phi_{\ep_0}(N).\]  So by the invariance of $\delta$ we have \[ \delta(N,\phi_{\ep_0}(N))=\delta\left(\psi_{\ep_0-\eta}(N),\psi_{\ep_0-\eta}(\phi_{\eta}(N))\right)=\delta(N,\phi_{\eta}(N)).\]  But $\delta(N,\phi_{\eta}(N))\leq \eta$ and $\eta\in (0,\ep_0]$ was arbitrary, so we have $\delta(N,\phi_{\ep_0}(N))=0$.  But $N\cap \phi_{\ep_0}(N)=\varnothing$, so by Lemma \ref{RNlemma}(v) we see that $R_N=\varnothing$.  Thus by Lemma \ref{RNlemma}(iii), $N$ is weightless.
\end{proof}

\subsection{Lagrangian submanifolds}\label{lagsec}
Having established results which allow us to show that the rigid locus $R_N$ is small in some cases, we now set about proving a result (Corollary \ref{lagopen} below) which can sometimes be used to show that $R_N$ is large.

Recall from \cite[Chapter X]{ALP} that a symplectic manifold $(M,\omega)$ is called \emph{geometrically bounded} if there exists an almost complex structure $\hat{J}$ and a complete Riemannian metric $\langle\cdot,\cdot\rangle$ on $M$ such that:\begin{itemize} \item There are constants $c_1,c_2>0$ such that for all $m\in M$ and  $v,w\in T_m M$ we have $\omega(v,\hat{J}v)\geq c_1\langle v,v\rangle$ and $|\omega(v,w)|^2\leq c_2\langle v,v\rangle\langle w,w\rangle$.
\item The Riemannian manifold $(M,\langle\cdot,\cdot\rangle)$ has sectional curvature bounded above and injectivity radius bounded away from zero.\end{itemize}  In particular such manifolds are tame in the sense of \cite[Chapter V]{ALP} and so satisfy the compactness theorems therein for $\omega$-tame almost complex structures which agree with $\hat{J}$ outside of a compact set.

The following result can be deduced from \cite[Theorem J]{FOOO} under suitable unobstructedness assumptions on $L$ and $L'$ and from results of \cite[Section 3.2.3.B]{BC06} when $L$ and $L'$ are Hamiltonian isotopic; however the general case does not seem to be in the literature.

\begin{theorem}\label{disptvs} Let $L$ and $L'$ be two compact Lagrangian submanifolds  of a geometrically bounded symplectic manifold $(M,\omega)$.  Assume that the intersection of $L$ and $L'$ is nonempty and transverse.  Then there is $\delta>0$ such that for any $\phi\in Ham(M,\omega)$ with $\phi(L)\cap L'=\varnothing$ we have $\|\phi\|\geq \delta$.
\end{theorem}

\begin{proof} Our argument is similar to that in \cite{Oh} (in which $L$ and $L'$, instead of being transverse, are equal). Suppose that $H\co [0,1]\times M\to\R$ is any compactly supported smooth function, whose Hamiltonian vector field at time $t\in [0,1]$ is given by $X_{H}(t,\cdot)$. Choose any  smooth family  $J=\{J_t\}_{t\in [0,1]}$ of almost complex structures with $J_0=J_1$, all of which  coincide outside a fixed compact set with some fixed almost complex structure $\hat{J}$ as in the definition of the geometrical boundedness of $(M,\omega)$.  Choose $\delta>0$ such that $\delta<\delta_J$ where $\delta_J$ is the minimum of: \begin{itemize} \item the smallest energy of a nonconstant $J_t$-holomorphic sphere as $t$ varies through $[0,1]$ \item the smallest energy of a nonconstant $J_0$-holomorphic disc with boundary on either $L$ or $L'$ \item the smallest energy of a nonconstant finite-energy map $u\co \R\times[0,1]\to M$ such that $\frac{\partial u}{\partial s}+J_t\frac{\partial u}{\partial t}=0$ and  $u(s,0)\in L$ and $u(s,1)\in L'$ for all $s\in \R$.  \end{itemize}  Of course, Gromov--Floer compactness (\cite[Chapters V, X]{ALP}, \cite[Proposition 2.2]{Fl}) implies that $\delta_J>0$, and that for any family of almost complex structures $J'$ sufficiently $C^{1}$-close to $J$ such that each $J'_t$ coincides with $\hat{J}$ outside a compact set we will have $\delta<\delta_{J'}$. 

For any $R>0$ let $\beta_R\co \R\to [0,1]$ be a smooth function such that $\beta_R(s)=1$ for $|s|\leq R$, $\beta_R(s)=0$ for $|s|\geq R+1$, and $s\beta'_{R}(s)\leq 0$ for all $s$.  

For any $\lambda\in [0,1]$ and $R>0$, consider solutions $u\co \R\times [0,1]\to M$ to the boundary value problem \begin{align}
\label{lR}\frac{\partial u}{\partial s}+ J_t\left(\frac{\partial u}{\partial t}-\lambda\beta_R(s)X_H(t,u(s,t))\right)&=0  \nonumber
\\ u(s,0),\,
u(s,1)&\in L' \end{align}

Since $\beta_R(s)=0$ for $|s|> R+1$ and since $L$ is transverse to $L'$, it follows as in the sentence after \cite[Proposition 2.2]{Fl} that for any finite-energy solution $u$ there will be points $p_{\pm}\in  L\cap L'$ so that $u(s,t)\to p_{\pm}$ uniformly in $t$ as $s\to\pm\infty$,  where the energy of $u$ is defined by $E(u)=\int_{\R\times[0,1]}\left|\frac{\partial u}{\partial s}\right|^{2}_{J}dsdt$.  So a finite-energy solution $u$ to (\ref{lR}) extends continuously to a map $u\co [-\infty,\infty]\times [0,1]\to M$ with $u([\infty,\infty]\times\{0\})\subset L$ and $u([\infty,\infty]\times\{1\})\subset L'$.  
Choose one point $p\in L\cap L'$.  From now on we only consider finite-energy solutions $u$ to (\ref{lR}) such that $u(s,t)\to p$ uniformly in $t$ both as $s\to-\infty$ and as $s\to +\infty$, so that  $s\mapsto u(s,\cdot)$ gives a loop in the space of paths from $L_0$ to $L_1$, and we moreover restrict attention to those $u$ such that this associated loop is homotopic to a constant.  Since $L_0$ and $L_1$ are Lagrangian, it is easy to see from Stokes' theorem that this homotopical assumption on $u$ implies that $\int_{\R\times [0,1]}u^*\omega=0$.  Now for any such $u$ which obeys (\ref{lR}) for  given values of $\lambda$ and $R$ we have the familiar energy estimate \begin{align*}
E(u)&=\int_{\R\times[0,1]}\left|\frac{\partial u}{\partial s}\right|^{2}_{J}dsdt=\int_{0}^{1}\int_{-\infty}^{\infty}\omega\left(\frac{\partial u}{\partial s},\frac{\partial u}{\partial t}-\lambda\beta_{R}(s)X_{H}(t,u(s,t))\right)dsdt 
\\&= \int_{\R\times [0,1]}u^*\omega-\int_{0}^{1}\int_{-\infty}^{\infty}\lambda\beta_R(s)d(H(t,\cdot))\left(\frac{\partial u}{\partial s}\right)dsdt 
\\&=-\int_{0}^{1}\int_{-\infty}^{\infty}\left(\frac{d}{ds}\left(\lambda\beta_R(s)H(t,u(s,t))\right)-\lambda\beta_{R}'(s)H(t,u(s,t))\right)dsdt
\\&\leq \lambda\int_{0}^{1}\left(\max_MH(t,\cdot)-\min_M H(t,\cdot)\right)dt=\lambda\|H\|.
\end{align*}

Here we use that, for all $t$, $\int_{-\infty}^{\infty}\frac{d}{ds}\left(\lambda\beta_R(s)H(t,u(s,t))\right)ds=0$ by the Fundamental Theorem of Calculus, while the assumed properties of $\beta_R$ ensure that \[ \int_{-\infty}^{0}\beta_{R}'(s)H(t,u(s,t))ds\leq \max_M H(t,\cdot)\quad \mbox{and}\quad \int_{0}^{\infty}\beta_{R}'(s)H(t,u(s,t))ds\leq -\min_MH(t,\cdot).\]

In particular the energy estimate above implies that the unique solution to (\ref{lR}) with the prescribed asymptotic and topological behavior for $\lambda=0$ is the constant solution $u(s,t)=p$.

Now suppose that our Hamiltonian $H\co [0,1]\times M\to\R$ obeys $\|H\|\leq \delta$.

For any $R>0$, and for any family  of $\omega$-compatible almost complex structures $J'=\{J'_t\}_{t\in[0,1]}$ with $J'_{1}=J'_{0}$, let $\mathcal{M}_{J',H}^{0,R}(p)$ denote the set of pairs $(\lambda,u)$ where $\lambda\in [0,1]$ and $u$ is a finite-energy solution to (\ref{lR}) for the given values of $\lambda$ and $R$, with $J'$ playing the role of the family of almost complex structures, such that $u$ is asymptotic at both ends to $p$ and  such that the associated loop of paths from $L$ to $L'$ is null-homotopic.  Standard arguments (essentially the same as those in \cite[Proposition 3.2]{Oh93}, \cite[p. 902]{Oh}) show that, for families of almost complex structures  $J'$ which are generic among those coinciding with $\hat{J}$ outside of a fixed precompact open set containing $L\cup L'$,
$\mathcal{M}_{J',H}^{0,R}(p)$ can be given the structure of a $1$-manifold with boundary where the boundary consists of the subsets corresponding to $\lambda=0$ and $\lambda=1$.  Moreover, provided that $J'$ is sufficiently close to $J$ this manifold with boundary is compact: indeed the only possible degenerations involve either bubbling of a holomorphic sphere or of a holomorphic disc with boundary on $L$ or $L'$, or else ``trajectory breaking'' involving a holomorphic strip $v\co \R\times S^1\to M$ with $v(\R\times\{0\})\subset L$ and $v(\R\times \{1\})\subset L'$.  But if $J'$ is sufficiently close to $J$ (so that what we previously denoted $\delta_{J'}$ is larger than $\delta$), then our energy estimate together with the fact that $\|H\|\leq \delta$ implies that the elements of $\mathcal{M}_{J',H}^{0,R}(p)$ have energy bounded above by a number smaller than $\delta_{J'}$, so that no bubbling or trajectory breaking can occur. 

As noted earlier, the part of the boundary of $\mathcal{M}_{J',H}^{0,R}(p)$ corresponding to $\lambda=0$ consists only of the constant map to $p$.   So since a compact $1$-manifold with boundary necessarily has an even number of boundary points, the part of the boundary of $\mathcal{M}_{J',H}^{0,R}(p)$ corresponding to $\lambda=1$ must be nonempty whenever $J'$ is sufficiently $C^1$-close to $J$.  Another application of Gromov compactness (taking the limit as $J'$ approaches $J$ and again using the energy bound to preclude bubbling and trajectory breaking) shows that the part of $\mathcal{M}_{J,H}^{0,R}(p)$ corresponding to $\lambda=1$ is also nonempty (even if $\mathcal{M}_{J,H}^{0,R}(p)$ is not itself a manifold).

Thus we have shown that, for any $R>0$, there is a solution $u\co \R\times [0,1]\to M$ to the $\lambda=1$ version of (\ref{lR}) asymptotic at both ends to $p$ whose associated loop of paths from $L$ to $L'$ is nullhomotopic.  Consequently the energies of all of these solutions are necessarily bounded above by $\|H\|\leq \delta$.  But then for any $R>0$ there must be $s_R\in [-R,R]$ such that the path $\gamma_R(t)=u(s_R,t)$ obeys \begin{equation}\label{gr}\int_{0}^{1}|\dot{\gamma}_R(t)-X_{H}(t,\gamma_R(t))|_{J}^{2}dt<\frac{\delta}{2R}.\end{equation} Morrey's inequality then bounds the $C^{1/2}$-norm of the $\gamma_R$, and hence the Arzela--Ascoli theorem yields a sequence $R_j\to\infty$ and a continuous path $\gamma\co [0,1]\to M$ such that $\gamma_{R_j}\to \gamma$ uniformly as $j\to\infty$ (so in particular $\gamma(0)\in L$ and $\gamma(1)\in L'$).   But then $X_H(t,\gamma_{R_j}(t))\to X_H(t,\gamma(t))$ uniformly in $t$, so by again applying (\ref{gr}) we see that the sequence $\{\dot{\gamma}_{R_j}\}_{j=1}^{\infty}$ is Cauchy in $L^2$.  Consequently $\gamma$ is the limit of $\gamma_{R_j}$ in the Sobolev space $W^{1,2}$, and in particular $\gamma$ has at least a weak derivative $\dot{\gamma}$ in $L^2$, which is equal to $t\mapsto X_{H}(t,\gamma(t))$.  But then since $\gamma$ is now known to be of class $W^{1,2}$ this latter function is also of class $W^{1,2}$, \emph{i.e.}, $\dot{\gamma}$ is of class $W^{1,2}$, and so $\gamma$ is of class $W^{2,2}$.  So by another application of Morrey's inequality $\gamma$ is $C^1$, and so is a genuine solution of the differential equation $\dot{\gamma}(t)=X_H(t,\gamma(t))$, satisfying $\gamma(0)\in L$, $\gamma(1)\in L'$. So where $\phi$ is the time-one map of $H$ the point $\gamma(1)$ lies in both $L'$ and $\phi(L)$.

This proves that any Hamiltonian diffeomorphism $\phi$ of Hofer norm at most $\delta$ necessarily satisfies $\phi(L)\cap L'\neq\varnothing$, as desired.
\end{proof}

\begin{cor}\label{lagcor} Let $(M,\omega)$ be a geometrically bounded symplectic manifold, $L\subset M$ a compact Lagrangian submanifold, and $U\subset M$ an open subset such that $L\cap U\neq\varnothing$.  Then there is $\delta>0$ such that if $\phi\in Ham(M,\omega)$ and $\|\phi\|<\delta$ then $\phi(L)\cap U\neq \varnothing$.
\end{cor}

\begin{proof} Let $B^{2n}(r)$ denote the standard symplectic ball of radius $r$ around the origin in $\C^{n}$, where $2n=\dim_{\R} M$, and let $S^1(r/2)\subset \C$ denote the circle of radius $r/2$ around the origin.  
In view of the Weinstein Neighborhood Theorem, there is $r>0$ and a Darboux chart $\psi\co V\to B^{2n}(r)$ around some point in $L$ such that $L\cap V=\psi^{-1}(\R^{n})$, where $V\subset U$ is an open subset.  Let $L'=\psi^{-1}\left((S^1(r/2))^n\right)$.  Then $L$ and $L'$ are Lagrangian submanifolds which meet each other transversely in $2^n$ points, with $L'\subset U$.  If $\phi\in Ham(M,\omega)$ has $\phi(L)\cap U=\varnothing$, then $\phi(L)\cap L'=\varnothing$, and so where $\delta$ is as in Theorem \ref{disptvs} we have $\|\phi\|\geq\delta$.
\end{proof}

\begin{cor}\label{lagopen}  Let $(M,\omega)$ be a geometrically bounded symplectic manifold, $N\subset M$ a closed subset, and $L$ a compact Lagrangian submanifold of $M$, which is contained in $N$.  Then $L\subset R_N$.
\end{cor}

\begin{proof} We must show that for any $x\in L$ and $\psi\in Ham(M,\omega)$ such that $\psi(x)\notin N$, it holds that $\psi\notin \bar{\Sigma}_N$.  

If $x\in L$, $\psi\in Ham(M,\omega)$, and $\psi(x)\notin N$, then $\psi(L)$ intersects the open subset $M\setminus N$, and so by Corollary \ref{lagcor} there is $\delta>0$ such that whenever $\|\phi\|<\delta$ we have $\phi(\psi(L))\cap (M\setminus N)\neq\varnothing$.  In particular since $L\subset N$ we have $\phi\circ\psi\notin \Sigma_N$ whenever $\|\phi\|<\delta$.  So the $\delta$-ball around $\psi$ is disjoint from $\Sigma_N$, proving that $\psi\notin \bar{\Sigma}_N$.
\end{proof}

\begin{remark} Note that in the case that $N=L$, this gives a new proof of Chekanov's theorem \cite{Ch00} that compact Lagrangian submanifolds of geometrically bounded symplectic manifolds are CH-rigid; this proof seems to be somewhat simpler than Chekanov's original one.  Actually this proof, unlike Chekanov's, can be extended to certain noncompact Lagrangian submanifolds $L$ of completions of Liouville domains such as the conormal bundles considered in \cite{Oh1} and, more generally, the Lagrangian submanifolds considered in \cite[(3.3)]{AS}; one just needs to have a maximum principle (such as the one proven in \cite[Section 7c]{AS}) in order to obtain compactness results for solutions of (\ref{lR}) when $L'$ (but perhaps not $L$) is compact
and $H$ is compactly supported, and then the proofs of Theorem \ref{disptvs} and Corollary \ref{lagopen} go through unchanged. 
\end{remark}

\begin{cor}\label{denserigid} Where $(M,\omega)$ is a geometrically bounded symplectic manifold, let $N\subset M$ be a closed subset such that there exists a dense subset $N_0\subset N$ so that for every $x\in N_0$ there is a compact Lagrangian submanifold $L_x\subset M$  so that $x\in L_x\subset N$.  Then $N$ is CH-rigid.
\end{cor}

\begin{proof} By Corollary \ref{lagopen} each of the Lagrangian submanifolds $L_x$ are contained in $R_N$, and so $N_0$ is contained in $R_N$, which is closed by Lemma \ref{RNlemma}(i). Thus $R_N=N$ since $N_0$ is dense in $N$, and so it follows from Lemma \ref{RNlemma}(ii) that $N$ is CH-rigid.
\end{proof}

\begin{remark} Note that if $N\subset M$ is a submanifold satisfying the hypothesis of Corollary \ref{denserigid} it is clear (independently of our other results) that $N$ is coisotropic: indeed for any $x\in N_0$ we have $T_xL_{x}\leq T_xN$ and so using that $L_x$ is Lagrangian we get a chain of inclusions $T_xN^{\omega}\leq T_x L_{x}^{\omega}=T_x L_x\leq T_xN$.  So $T_xN^{\omega}\leq T_xN$ throughout a dense subset of $N$, and so indeed throughout all of $N$.  Of course this is consistent with the conclusion of Corollary \ref{denserigid} together with Corollary \ref{ncdeg}.  This argument also shows that for any $x\in N_0$ the Lagrangian submanifold  $L_x$ must contain the entire leaf of the characteristic foliation through $x$.
\end{remark}

\begin{ex}\label{ellipsoid}
For a tuple $\vec{a}=(a_1,\ldots,a_n)\in (0,\infty)^{n}$ let \[ E_{\vec{a}}=\left\{(x_1,\ldots,x_{2n})\in\R^{2n}\left| \sum_{i=1}^{n}\left(\frac{x_{i}^{2}+x_{n+i}^{2}}{a_i}\right)=1\right.\right\}\] (thus $E_{\vec{a}}$ is the boundary of the standard symplectic ellipsoid having cross-sections of capacity $\pi a_i$).  Of course $E_{\vec{a}}$ is CH-rigid by Theorem \ref{hypthm} simply by virtue of being a codimension-one submanifold. Considering instead products $E_{\vec{a}^{(1)}}\times\cdots\times  E_{\vec{a}^{(k)}}\subset \R^{2n_1+\cdots+2n_k}$ for $\vec{a}^{(j)}\in (0,\infty)^{n_j}$, we claim that it continues to hold that
$E_{\vec{a}^{(1)}}\times\cdots\times  E_{\vec{a}^{(k)}}$ is always CH-rigid.

Indeed, for any given $\vec{a}\in(0,\infty)^{n}$, a dense subset of $E_{\vec{a}}$ is foliated by compact Lagrangian submanifolds: for any positive numbers $b_1,\ldots,b_n$ such that $\sum_{i=1}^{n}\frac{b_i}{a_i}=1$,  the submanifold \[ L_{\vec{b}}=\left\{(x_1,\ldots,x_{2n})\in\R^{2n}\left| (\forall i) (x_{i}^{2}+x_{n+i}^{2}=b_i)\right.\right\}\] is a Lagrangian torus in $\R^{2n}$ which is contained in $E_{\vec{a}}$, and any point in the dense subset of $E_{\vec{a}}$ consisting of $(x_1,\ldots,x_{2n})$ such that every $x_{i}^{2}+x_{n+i}^{2}$ is nonzero will belong to one of the $L_{\vec{b}}$.  Taking products $L_{\vec{b}^{(1)}}\times\cdots\times L_{\vec{b}^{(k)}}$ gives a foliation of a dense subset of $E_{\vec{a}^{(1)}}\times\cdots\times  E_{\vec{a}^{(k)}}$ by Lagrangian tori, and so Corollary \ref{denserigid} shows that $E_{\vec{a}^{(1)}}\times\cdots\times  E_{\vec{a}^{(k)}}$ is CH-rigid.
\end{ex}

\subsection{The instability of small rigid loci}\label{unstable}

It follows from Lemma \ref{rnp} that if $N$ is a submanifold of $(M,\omega)$ (which we will implicitly assume to have dimension greater than $\frac{1}{2}\dim M$) then the rigid locus $R_N$ cannot be contained in a submanifold of $N$ of dimension less than $\dim M-\dim N$, unless of course $N$ is weightless so that $R_N$ is empty.  On the other hand (assuming that $(M,\omega)$ is geometrically bounded) if there is a compact Lagrangian submanifold $L$ of $M$ contained in $N$ such that at every point $x\in N\setminus L$ we have $T_{x}N^{\omega}\not\subset T_xN$, then we will have $R_N=L$.  It is not clear at this point whether a nonempty $R_N$ can ever be contained in a submanifold of dimension less than $\frac{1}{2}\dim M$; however what we will do presently shows that, if this ever happens for a compact $N$, then it is an ``unstable'' phenomenon, in that it disappears under taking a product with $S^1\subset\R^2$.

We adopt some notation relating to such stabilizations.  If $(M,\omega)$ is a symplectic manifold and $N\subset M$ is any subset, consider the symplectic manifold $\left(\R^2\times M,\Omega=(dx\wedge dy)\oplus \omega\right)$, and define \[ \hat{N}=S^1\times N\subset \R^2\times M.\]

Recall that if $N\subset M$ is any closed subset the \emph{displacement energy} of $N$ in $M$ is \[ e(N,M)=\inf\{\phi\in Ham(M,\omega)|\phi(N)\cap N=\varnothing\}.\]

\begin{theorem}\label{stable}  Suppose that $N\subset M$ is a compact subset with the property that $e(R_N,M)=0$.  Then the subset $\hat{N}=S^1\times N\subset \R^2\times M$ is weightless.
\end{theorem}

\begin{proof} We begin with a lemma:

\begin{lemma}\label{betalemma}  For any $\ep>0$ and $R>0$  there is $\phi\in Ham(\R^2\times M,\Omega)$ such that: \begin{itemize}
\item $\|\phi\|<\ep$
\item For any $(x,y,m)\in \R^2\times M$ with $|x|+|y|\leq R$ the first coordinate of $\phi(x,y,m)$, denoted $x'$, has $x-3\leq x'\leq x+3$, and the second coordinate of $\phi(x,y,m)$ is equal to $y$. 
\item There is a neighborhood $W$ of $R_N$ in $M$ such that if $m\in W$ and $|x|+|y|\leq R$ then $\phi(x,y,m)$ has its first coordinate $x'$ equal to $x+3$.  Moreover $\bar{W}$ is compact.
\end{itemize}
\end{lemma}

\begin{proof}[Proof of Lemma \ref{betalemma}] First let $\eta\in Ham(M,\omega)$ be such that $\|\eta\|<\frac{\ep}{2}$ and $\eta(R_N)\cap R_N=\varnothing$, as we can do by the assumption in the theorem that $e(R_N,M)=0$.  Now let $\gamma\co M\to [0,1]$ be a compactly  supported smooth function such that for some neighborhood $W$ of $R_N$ with compact closure it holds that $\gamma|_{W}=0$ and $\gamma^{-1}(\{1\})=\overline{\eta(W)}$.  Let $H\co \R^2\times M\to \R$ be a smooth function such that $H(x,y,m)=-3y\gamma(m)$ whenever $|x|+|y|\leq R+3$ and let $\psi\co \R^2\times M\to\R^2\times M$ be the time-one map of $H$.  Finally, choose $\tilde{\eta}\in Ham(\R^2\times M,\Omega)$ such that $\|\tilde{\eta}\|\leq \|\eta\|$ and for all $(x,y,m)\in \R^2\times M$ with $|x|+|y|<R+3$ we have $\tilde{\eta}(x,y,m)=(x,y,\eta(m))$.  Such a $\tilde{\eta}$ can easily be constructed as the time-one map of a Hamiltonian obtained from the Hamiltonian generating $\eta$ by pulling back via the projection and then multiplying by a suitable cutoff function.

Our map $\phi\in Ham(\R^2\times M,\Omega)$ will be given by the formula \[ \phi=  \psi^{-1}\circ \tilde{\eta}^{-1}\circ \psi \circ\tilde{\eta}.\]  By the triangle inequality and the invariance of the Hofer norm under conjugation and inversion we see that $\|\phi\|\leq 2\|\tilde{\eta}\|$, and by assumption $\|\tilde{\eta}\|\leq \|\eta\|<\frac{\ep}{2}$; thus $\|\phi\|<\ep$.

 Now the Hamiltonian vector field of $H$ is given within $\{|x|+|y|\leq R+3\}\times M$ by $3\gamma(m)\frac{\partial}{\partial x}-3yZ_{\gamma}$, where $Z_{\gamma}$ is the Hamiltonian vector field of $\gamma$ on $M$, trivially pushed forward to $\R^2\times M$.  So (at least for $|x|+|y|\leq R$) $\psi$ does not change the $y$ coordinate and (since $0\leq\gamma\leq 1$) changes the $x$ coordinate by an amount between $0$ and $3$.  Since $\tilde{\eta}$ does not affect the $\R^2$ factor within $\{|x|+|y|\leq R+3\}\times M$, the second statement of the lemma follows directly.
 
   For the third statement, let $(x,y,m)\in \R^2\times W$ with $|x|+|y|\leq R$, where $W$ is as in the first paragraph of the proof.  Then $\gamma(\eta(m))=1$, and so (using that $d\gamma(Z_{\gamma})=0$) where $(x_1,y_1,m_1)=\psi\circ\tilde{\eta}(x,y,m)=\psi(x,y,\eta(m))$ we will have $x_1=x+3$, $y_1=y$, and $\gamma(m_1)=\gamma(\eta(m))=1$.  So since $\tilde{\eta}(x,y,m)=(x,y,\eta(m))$ for $|x|+|y|<R+3$ we have $\tilde{\eta}^{-1}(x_1,y_1,m_1)=(x+3,y,\eta^{-1}(m_1))$.  Now by our construction of $\gamma$ and $W$ the fact that $\gamma(m_1)=1$ implies that $\gamma^{-1}(m_1)\in\bar{W}$ and hence that $\gamma(\eta^{-1}(m_1))=0$, and so the first coordinate of  $\psi^{-1}(x+3,y,\eta^{-1}(m_1))$ will be $x+3$.   
\end{proof}

Lemma \ref{betalemma} has the following consequence.

\begin{lemma}\label{zetalemma} Again assuming that $N\subset M$ is compact and $e(R_N,M)=0$, for any positive integer $n$ there is an element $\zeta_n\in Ham(\R^2\times M,\Omega)$ such that $\|\zeta_n\|<1/n$ and $\zeta_n(\hat{N})\subset \{(x,y,m)\in \R^2\times M|x\geq 2\}$.
\end{lemma}

\begin{proof}[Proof of Lemma \ref{zetalemma}] Fix $n$ and some  number $R>15$ and apply Lemma \ref{betalemma} with $\ep=\frac{1}{2n}$ to obtain an element $\phi\in Ham(\R^2\times M,\Omega)$ and an open set $W$ satisfying the indicated properties.  Choose a neighborhood $V$ of $\bar{W}$ and a smooth compactly supported function $\alpha\co M\to [0,1]$ such that $\alpha^{-1}(0)$ has interior which contains $R_N$, and $W=V\cap\alpha^{-1}\left([0,1)\right)$. Note that this implies, via an easy connectedness argument, that any path in $\alpha^{-1}\left([0,1)\right)$ which begins in $W$ also ends in $W$.   Let $H\co \R^2\times M\to \R$ be a compactly supported smooth function with $H(x,y,m)=-6y\alpha(m)$ wherever $|x|+|y|<R+6$.  Let $\psi$ be the time-one map of $H$.  Then $\psi$ obeys the following properties: \begin{itemize}
\item[(i)] $\psi\in Ham^c\left((\R^2\times M)\setminus (S^1\times R_N)\right)$ (\emph{i.e.}, $\psi$ is generated by a Hamiltonian with compact support in $(\R^2\times M)\setminus (S^1\times R_N)$).
\item[(ii)] For any $(x,y,m)\in S^1\times N$ the first coordinate of $\phi\circ\psi(x,y,m)$ is at least equal to $x+3$.
\end{itemize}

Indeed, (i) is obvious, while for (ii), the first coordinate of $\phi(\psi(x,y,m))$ will be no smaller than $3$ less than that of $\psi(x,y,m)$. (Here we use the second item in Lemma \ref{betalemma}, which is easily seen to apply to the point $\psi(x,y,m)$ by the definition of $\psi$ and the facts that $(x,y)\in S^1$ and $R>15$.) The first coordinate of $\psi(x,y,m)$ will be equal to $x+6$ unless $\alpha(m)<1$, and will be at least equal to $x$ in any event. Now if $\alpha(m)<1$, then $m\in W$ (as we are assuming $(x,y,m)\in S^1\times N$). Moreover since $\alpha$ is constant along the Hamiltonian flow of $H$, 
writing $\psi(x,y,m)=(x',y',m')$ we will have  $m'\in W$ (using our earlier remark that a path in $\alpha^{-1}\left([0,1)\right)$ which begins in $W$ also ends in $W$), and so the first coordinate of $\phi(x',y',m')=\phi\circ\psi(x,y,m)$ will be equal to $x'+3\geq x+3$ by the last property in Lemma \ref{betalemma}.  So in any case (ii) will hold.  

We claim that (by virtue of (i) above) $\psi\in \bar{\Sigma}_{\hat{N}}$.  First of all note that $R_{\hat{N}}\subset S^1\times R_N$.  Indeed if $(x,y,m)\in \hat{N}=S^1\times N$ with $m\in N\setminus R_N$, so that there is $g\in \bar{\Sigma}_N$ with $g(m)\notin N$, then it is easy to find an element of $\bar{\Sigma}_{\hat{N}}$ which restricts to a neighborhood of $S^1\times N$ as $(id_{\R^2}\times g)$ and hence moves $(x,y,m)$ off of $\hat{N}$, proving that $(x,y,m)\notin R_{\hat{N}}$.  Now just as in the proof of Lemma \ref{RNlemma}(iii), if $p\in (\R^2\times M)\setminus (S^1\times R_N)$, so that in particular $p\notin R_{\hat{N}}$, then we can find a neighborhood $V_p$ of $p$ in  $(\R^2\times M)\setminus (S^1\times R_N)$ and an element $g_p\in \bar{\Sigma}_{\hat{N}}$ so that $g_p(V_p)\cap \hat{N}=\varnothing$.  So if $\eta\in Ham^c(V_p)$ then $g_p\circ\eta\circ g_{p}^{-1}\in \Sigma_{\hat{N}}$, implying that $\eta\in \bar{\Sigma}_{\hat{N}}$ since $\bar{\Sigma}_{\hat{N}}$ is a group containing both $\Sigma_{\hat{N}}$ and $g_p$.  Thus $(\R^2\times M)\setminus (S^1\times R_N)$ is covered by open sets $V_p$ with $Ham^c(V_p)\leq \bar{\Sigma}_{\hat{N}}$, which by Banyaga's fragmentation lemma implies that $Ham^c\left((\R^2\times M)\setminus (S^1\times R_N)\right)\leq \bar{\Sigma}_{\hat{N}}$.  So by (i) we indeed have 
$\psi\in \bar{\Sigma}_{\hat{N}}$.  

Accordingly we can choose $\xi\in \Sigma_{\hat{N}}$ so that $\|\psi\circ\xi^{-1}\|<\frac{1}{2n}$.    Set $\zeta_n=\phi\circ\psi\circ\xi^{-1}$.  Since $\xi$ belongs to the stabilizer $\Sigma_{\hat{N}}$ we have $\zeta_n(\hat{N})=\phi\circ\psi(\hat{N})$, which is contained in $\{x\geq 2\}$ by (ii).  Moreover $\|\zeta_n\|\leq \|\phi\|+\|\psi\circ\xi^{-1}\|<\frac{1}{2n}+\frac{1}{2n}$, as desired.
\end{proof}

We now complete the proof of Theorem \ref{stable}.  \begin{claim}\label{clm} $\delta(\hat{N},\zeta_n(\hat{N}))$ is independent of $n$. \end{claim}  \begin{proof}[Proof of Claim \ref{clm}] For any $T>0$ let $\rho_T$ denote the translation $(x,y,m)\mapsto (x+T,y,m)$.  Given positive integers $n_1,n_2$, we know that for $i=1,2$, $\zeta_{n_i}(\hat{N})$ is a compact submanifold of $\R^2\times M$ contained in $[2,\infty)\times \R\times N$, so choose a compact subset $K\subset M$ and a number $A\gg 1$ so that $\zeta_{n_i}(\hat{N})\subset [2,A]\times [-A,A]\times int(K)$ for $i=1,2$.  Let $H\co \R^2\times M\to\R$ be a smooth function whose support is compact and contained in $[1.5,\infty)\times\R\times M$, such that the restriction of $H$ to $[2,A]\times [-A,A+T]\times K$ 
coincides with the function $(x,y,m)\mapsto -Ty$.  Then the time-one map $\phi_{H}^{1}$ will obey $\phi_{H}^{1}(\hat{N})=\hat{N}$ while, for $i=1,2$, $\phi_{H}^{1}(\zeta_{n_i}(\hat{N}))=\rho_T(\zeta_{n_i}(\hat{N}))$.  Consequently we have, for $i=1,2$ and any $T>0$, \begin{equation}\label{transdelta} \delta(\hat{N},\zeta_{n_i}(\hat{N}))=\delta\left( \phi_{H}^{1}(\hat{N}), \phi_{H}^{1}(\zeta_{n_i}(\hat{N}))\right)=\delta\left(\hat{N},\rho_T(\zeta_{n_i}(\hat{N}))\right).\end{equation}

Now we have \[ \left(\rho_T\circ(\zeta_{n_2}\circ\zeta_{n_1}^{-1})\circ\rho_{T}^{-1}\right)(\rho_T(\zeta_{n_1}(\hat{N})))=\rho_T(\zeta_{n_2}(\hat{N})).\]  The Hamiltonian diffeomorphism $\zeta_{n_2}\circ\zeta_{n_1}^{-1}$ is compactly supported; denoting the support of $\zeta_{n_2}\circ\zeta_{n_1}^{-1}$ by $L$, the support of $\rho_T\circ(\zeta_{n_2}\circ\zeta_{n_1}^{-1})\circ\rho_{T}^{-1}$ will be $\rho_{T}(L)$, which is disjoint from $\hat{N}$ if $T$ is sufficiently large.  Hence the invariance of $\delta(\cdot,\cdot)$ under simultaneous action on both entries by the symplectomorphism
$\rho_T\circ(\zeta_{n_2}\circ\zeta_{n_1}^{-1})\circ\rho_{T}^{-1}$ gives, for $T\gg 1$, \[ \delta(\hat{N},\rho_{T}(\zeta_{n_1}(\hat{N})))=
\delta(\hat{N},\rho_{T}(\zeta_{n_2}(\hat{N}))).\]  Combining this with (\ref{transdelta}) evidently gives \[ \delta(\hat{N},\zeta_{n_1}(\hat{N}))=\delta(\hat{N},\zeta_{n_2}(\hat{N})),\] confirming Claim \ref{clm}\end{proof}

Now recalling from Lemma \ref{zetalemma} that $\|\zeta_n\|<1/n$, we evidently have $\delta(\hat{N},\zeta_n(\hat{N}))<1/n$, so the fact that 
$\delta(\hat{N},\zeta_n(\hat{N}))$ is independent of $n$ forces us to have $\delta(\hat{N},\zeta_n(\hat{N}))=0$ for all $n$.    So by 
Lemma \ref{RNlemma}(v), $R_{\hat{N}}\subset \hat{N}\cap \zeta_n(\hat{N})$.  But of course  $\hat{N}\cap \zeta_n(\hat{N})=\varnothing$, so $R_{\hat{N}}=\varnothing$; by Lemma \ref{RNlemma}(iii) this completes the proof of Theorem \ref{stable}.
\end{proof}

\section{Coisotropic submanifolds}\label{coisosect}

We now use the foregoing results to prove CH-rigidity for various classes of coisotropic submanifolds in geometrically bounded symplectic manifolds; in particular this will yield Theorem \ref{coisomain}.  Of course, hypersurfaces are coisotropic, as are Lagrangian submanifolds, so Theorem \ref{hypthm} and \cite{Ch00} already address two significant classes.  If $N$ is coisotropic, we have a distribution $TN^{\omega}$ on $N$ of rank $\dim M-\dim N$; recall that the fact that $\omega$ is closed implies that this distribution is integrable and so generates a foliation of $N$ (the ``characteristic foliation'').

The coisotropic submanifold $N$ of $(M,\omega)$ is called \emph{regular} if the sense that the characteristic foliation of $N$ is given by the fibers of a submersion. As noted in \cite[Lemma 24]{Zi}, the regularity of $N$ is equivalent to the statement that the leaf relation \[ R=\{(x,y)\in N\times N| x\mbox{ and }y \mbox{ are on the same leaf of the characteristic foliation}\}\] is a submanifold which is closed as a subset of $N\times N$.

Corollary \ref{denserigid} leads to the conclusion that a variety of coisotropic submanifolds, including regular ones, are CH-rigid:

\begin{theorem}\label{regrigid} Suppose that a coisotropic submanifold $N$ of the geometrically bounded symplectic manifold $(M,\omega)$ is regular (with the characteristic foliation having compact leaves), or else is given by $N=J^{-1}(\eta)$ for an equivariant moment map $J\co M\to \mathfrak{g}^{*}$ associated to a Hamiltonian action of a compact Lie group $G$ on $M$, where $\eta\in\mathfrak{g}^{*}$ is fixed by the coadjoint action of $G$ on $\mathfrak{g}^{*}$ and is a regular value for $J$.  Then $N$ is CH-rigid.
\end{theorem}

\begin{proof} The basic observation is the following:

\begin{lemma}\label{prelag}  Let $(M_0,\omega)$ be a symplectic manifold and let $N_0\subset M_0$ be a coisotropic submanifold such that for some symplectic manifold $(Z,\sigma)$ there is a proper surjective submersion $\pi\co N_0\to Z$ such that where $i\co N_0\to M_0$ is the inclusion we have $i^*\omega=\pi^*\sigma$.  Then for every $x\in N_0$ there is a compact Lagrangian submanifold $L_x$ of $M_0$ so that $x\in L_x\subset N_0$.
\end{lemma}

\begin{proof}[Proof of Lemma \ref{prelag}]  To construct $L_x$, let $\Lambda_x\subset Z$ be a Lagrangian torus containing $\pi(x)$ and contained in a  Darboux chart around $\pi(x)$.  Then set $L_x=\pi^{-1}(\Lambda_x)$.  Then $L_x$ is a compact submanifold of $N_0$ (since $\pi$ is a proper submersion), and clearly $x\in L_x$, so we need only check that $L_x$ is Lagrangian.  If $y\in L_x$ and $v,w\in T_yL_x$ then since $i^*\omega=\pi^*\sigma$ we have $\omega(v,w)=\sigma(\pi_*v,\pi_*w)=0$ since $\pi_*v$ and $\pi_*w$ are both tangent to the Lagrangian submanifold $\Lambda_x$ of $Z$.  So it only remains to show that $\dim L_x=\frac{1}{2}\dim M_0$.  

To see this, note that the fact that $N_0$ is coisotropic together with the fact that $i^*\omega=\pi^*\sigma$ where $\sigma$ is nondegenerate implies that for all $x\in N_0$ we have $T_{x}N_{0}^{\omega}=\ker(\pi_*)_x$.  Equating the dimensions of these two vector spaces shows that $\dim M_0-\dim N_0=\dim N_0-\dim Z$, \emph{i.e.}, $\dim Z=2\dim N_0-\dim M_0$.  So since $\Lambda_x\subset Z$ is Lagrangian, \begin{align*} \dim L_x&=\dim \Lambda_x+(\dim N_0-\dim Z)=\frac{1}{2}\dim Z+(\dim M_0-\dim N_0)\\&=\left(\dim N_0-\frac{1}{2}\dim M_0\right)+\left(\dim M_0-\dim N_0\right)=\frac{1}{2}\dim M_0.\qedhere\end{align*}
\end{proof}

Resuming the proof of Theorem \ref{regrigid}, if $N$ is regular then a standard argument (see, \emph{e.g.}, \cite[Lemma 5.35]{MS}) shows that, where $\pi\co N\to Z$ is the submersion whose fibers are the leaves of the characteristic foliation,  a symplectic form $\sigma$ may be constructed on $Z$ which obeys $\pi^*\sigma=\omega|_N$, and so we can apply Lemma \ref{prelag} and Corollary \ref{denserigid} to prove Theorem \ref{regrigid} in this case.

We now turn to the other case in Theorem \ref{regrigid}, in which $N=J^{-1}(\eta)$ where $J\co M\to\mathfrak{g}^*$ is an equivariant moment map for a Hamiltonian action of a compact Lie group $G$, and $\eta\in\mathfrak{g}^*$ is a regular value of $J$ which is fixed by the coadjoint action of $G$ on $\mathfrak{g}^*$.

 The assumption that $\eta$ is a regular value of $J$ implies that $J^{-1}(\eta)$ is a submanifold upon which $G$ acts locally freely (see \cite[Proposition 1.1.2]{MMOPR}), while  the assumption that $\eta$ is fixed by the coadjoint action implies that $J^{-1}(\eta)$ is coisotropic.  Although the action of $G$ on $J^{-1}(\eta)$ might not be free, one still has a stratification of $J^{-1}(\eta)$ by orbit type (\emph{i.e.}, by conjugacy classes of stabilizers); this stratification has a unique top stratum $N_0\subset J^{-1}(\eta)$ (the ``principal orbit stratum'') which is open and dense in $J^{-1}(\eta)$ (see, \emph{e.g.}, \cite[Theorem 5.9]{SL}).  Although $G$ might still not act freely on $N_0$, \cite[Theorem 1.4.2]{MMOPR} and \cite[Theorem 2.1]{SL} show  that under our assumptions it holds both that $Z_0=N_0/G$  has a unique smooth structure so that $\pi\co N_0\to Z_0$ is a submersion, and that $Z_0$ admits a symplectic structure such that the projection $\pi$ obeys the requirements of Lemma \ref{prelag}.  Since $N_0\subset J^{-1}(\eta)$ is dense, it therefore follows from Corollary \ref{denserigid} that $J^{-1}(0)$ is CH-rigid.
\end{proof}

\begin{remark}  Let us consider again the products of ellipsoids $E_{\vec{a}^{(1)}}\times\cdots\times  E_{\vec{a}^{(k)}}$ which were shown to be CH-rigid in Example \ref{ellipsoid} using Corollary \ref{denserigid}.
Depending on the numbers $a_{i}^{(j)}\in (0,\infty)$, the coisotropic submanifold $E_{\vec{a}^{(1)}}\times\cdots\times  E_{\vec{a}^{(k)}}$  may or may not satisfy the hypotheses of Theorem \ref{regrigid}.  If for each $j$ we have $a_{1}^{(j)}=\cdots=a_{n_j}^{(j)}$, so that each $E_{\vec{a}^{(j)}}$ is a sphere of radius $\sqrt{a_{1}^{(j)}}$, then the characteristic foliation of $E_{\vec{a}^{(1)}}\times\cdots\times  E_{\vec{a}^{(k)}}$ will just be the vertical foliation given by the product of the Hopf fibrations $E_{\vec{a}^{(j)}}\to \mathbb{C}P^{n_j-1}$; thus in this case  $E_{\vec{a}^{(1)}}\times\cdots\times  E_{\vec{a}^{(k)}}$ is regular.  If instead it only holds that for each $j$ the ratios $\frac{{a}_{i_1}^{(j)}}{a_{i_2}^{(j)}}$ are rational, so that there are some $\lambda_j\in (0,\infty)$ and $m_{ij}\in \mathbb{Z}_+$ so that $a_{i}^{(j)}=m_{ij}\lambda_j$, then it is not difficult to see that $E_{\vec{a}^{(1)}}\times\cdots\times  E_{\vec{a}^{(k)}}$ is the preimage of a regular value of the moment map for a Hamiltonian $T^k$-action on $\R^{2\sum n_i}$, so that Theorem \ref{regrigid} again applies.  However when the $a_{i}^{(j)}$ are rationally independent Theorem \ref{regrigid} does not seem to apply, demonstrating the greater generality the situations covered by Corollary \ref{denserigid}.
\end{remark}

In a different direction, some coisotropic submanifolds $N$ can be shown to be CH-rigid along the following lines: one shows that if $N$ were \emph{not} CH-rigid, then its rigid locus would have to be suitably ``small,'' and then deduces  from
Theorem \ref{stable} (or from a simpler argument) that this would contradict known rigidity properties for $N$ or for the stabilization $\hat{N}$.  

The basic observation is that $R_N\subset N$ is a closed subset which is invariant under the action of the stabilizer $\Sigma_N$ on $N$, and this imposes significant restrictions on $R_N$.  As a simple special case, for any closed subset $N$, coisotropic or not, on which $\Sigma_N$ acts transitively, it must hold that $N$ is either weightless or CH-rigid, and in some cases when $N$ is coisotropic results such as those in \cite{Gi} or \cite{U} can be used to rule out the former alternative.  

So we now consider the action of the $\Sigma_N$ on a coisotropic submanifold $N$.  Note that any $\phi\in \Sigma_N$ obeys, for each $x\in N$, $\phi_*T_{x}N^{\omega}=T_{\phi(x)}N^{\omega}$, in view of which $\phi$ permutes the leaves of the characteristic foliation.  In particular if not all leaves of the characteristic foliation are diffeomorphic then $\Sigma_N$ will not act transitively on $N$.  (If $N$ happens to be regular, on the other hand, then one can show that $\Sigma_N$ acts transitively on $N$, but of course this case is already covered by Theorem \ref{regrigid}.) 
We do in any case have the following:

\begin{prop}\label{lfws}  Let $N$ be a coisotropic submanifold of $(M,\omega)$. Then where $G$ is the subgroup of $\Sigma_N$ consisting of Hamiltonian diffeomorphisms of $M$ which preserve each leaf of the characteristic foliation, $G$ acts transitively on every leaf of the characteristic foliation.
\end{prop}

\begin{proof} Write $E=TN^{\omega}$ (so since $N$ is coisotropic $E\subset TN$), choose a Riemannian metric $h$ on $N$, and let $\Pi_h\co TN\to E$ be the orthogonal projection induced by $h$.  On the total space of the vector bundle $\pi\co E^*\to N$ define a $1$-form $\theta_h\in \Omega^1(E^*)$ by (for $x\in N$,  $p\in E^{*}_{x}$, and $v\in T_{(x,p)}E^*$) \[ (\theta_h)_{x,p}(v)=p(\Pi_h(\pi_*v))\] and define a $2$-form $\Omega$ on $E^*$ by \[ \Omega=\pi^*(\omega|_N)+d\theta_h.\]

As seen in \cite[Proposition 3.2]{Ma},  $\Omega$ restricts symplectically to a neighborhood $U$ of the zero section $N\subset E^*$ and $\Omega|_N=\omega|_N$; moreover there is a symplectomorphism from this neighborhood $U$ of $N\subset E^*$ to a neighborhood of $N$ in the original symplectic manifold $M$, restricting as the identity on $N$.  Consequently it suffices to prove  that there is a compactly supported Hamiltonian diffeomorphism of $(U,\Omega)$ which preserves each leaf of the characteristic foliation of the zero section $N$ and maps $x$ to $y$, where $x$ and $y$ are any given points of $N$ lying on the same leaf $\Lambda\subset N$.

To do so, choose a smooth path $\gamma\co [0,1]\to \Lambda$, so in particular $\gamma'(t)\in E_{\gamma(t)}$ for all $t$.  Let $V_t$ be a smooth one-parameter family of vector fields on $N$ such that $V_t(\gamma(t))=\gamma'(t)$ for all $t$, and $V_t(x)\in E_x$ for all $t,x$.  Define a function $H\co [0,1]\times E^*\to \R$ by $H(t,x,p)=p(V_t(x))$ for $t\in [0,1]$, $x\in N$, and $p\in E^{*}_{x}$.  

Along the zero section $N$, we have a canonical splitting $TE^{*}|_N\cong TN\oplus E^*$.  In terms of this splitting, and writing $H_t(x,p)=H(t,x,p)$, we see that $(dH_t)_{(x,0)}(v,\alpha)=\alpha(V_t(x))$ at any point $(x,0)$ on the zero-section, for all $v\in T_xN$ and $\alpha\in E^*$.  Meanwhile the vector field given in terms of the splitting $TE^{*}|_N\cong TN\oplus E^*$ by $(V_t,0)$ obeys, for $v\in T_xN$ and $\alpha\in E^*$, \[ \Omega_{(x,0)}((v,\alpha),(V_t,0))=(d\theta_h)_{(x,0)}((0,\alpha),(V_t,0))=\alpha(V_t(x))=(dH_t)_{(x,0)}(v,\alpha)\] (where we have used that $\iota_{V_t}(\omega|_N)=0$).  This shows that the restriction of the Hamiltonian vector field of $H_t$ to the zero section $N$ coincides with the vector field $V_t$, which was chosen to be tangent to the characteristic foliation and to the given curve $\gamma$ contained in one of the leaves.  Consequently, after cutting off $H_t$ to be compactly supported in $U$, we obtain as its time-one flow a Hamiltonian diffeomorphism $\phi$ of $U$ which preserves the leaves of the characteristic foliation on the zero section $N$ and such that $\phi(\gamma(0))=\gamma(1)$.  Since $\gamma(0)$ and $\gamma(1)$ may be chosen arbitrarily within the same leaf this proves the result.
\end{proof}

\begin{cor}\label{densealt} Let $N$ be a coisotropic submanifold of the symplectic manifold $(M,\omega)$ which is closed as a subset, and let $\Lambda$ be a leaf of the characteristic foliation of $N$ which is dense in $N$.  If $N$ is not CH-rigid then $\Lambda\cap R_N=\varnothing$.
\end{cor}

\begin{proof} If on the contrary we had some $x\in \Lambda\cap R_N$ then since we have $\phi(R_N)=R_N$ for all $\phi\in \Sigma_N$ it follows from Proposition \ref{lfws} that $\Lambda\subset R_N$.  So since $R_N\subset N$ is closed and $\Lambda$ is assumed dense in $N$ we obtain $R_N=N$.  Now use Lemma \ref{RNlemma}(ii).
\end{proof}

\begin{cor} Let $N$ be a compact coisotropic submanifold of the symplectic manifold $(M,\omega)$, and suppose that there is a closed subset $S\subset N$ such that $e(S,M)=0$ and such that for every $x\in N\setminus S$ the leaf of the characteristic foliation containing $x$ is dense in $N$.  If $N$ is not CH-rigid then the stabilization $\hat{N}\subset M\times \R^2$ must be weightless.
\end{cor}

\begin{proof}
If $N$ is not CH-rigid, then Corollary \ref{densealt} shows that $(N\setminus S)\cap R_N=\varnothing$, \emph{i.e.} that $R_N\subset S$.  So $e(R_N,M)\leq e(S,M)=0$, and so Theorem \ref{stable} shows that $\hat{N}$ is weightless.
\end{proof}

Recall  that a codimension-$k$ coisotropic submanifold $N\subset (M^{2n},\omega)$ is called 
\emph{stable} if there are $1$-forms $\alpha_1,\ldots,\alpha_k\in \Omega^1(N)$ 
such that for each $i$ we have $\ker(\omega|_N)\subset \ker d\alpha_i$ and such that $\alpha_1\wedge\cdots\wedge \alpha_k\wedge(\omega|_N)^{n-k}$ is a volume form on $N$.  
See \cite[Section 2.1]{Gi} for introductory remarks about stable coisotropic submanifolds.  Note in particular that if $N_1$ is a coisotropic submanifold of $M_1$, and $N_2$ is a stable coisotropic submanifold of $M_2$, then $N_1\times N_2$ is a stable coisotropic submanifold of $M_1\times M_2$.  Since $S^1\subset \mathbb{R}^2$ is stable, this in particular implies that if $N\subset M$ is stable then so is $\hat{N}\subset M\times \R^2$.  The following corollary now implies Theorem \ref{coisomain}(ii).

\begin{cor}\label{posdisp} Let $N\subset M$ be a compact stable coisotropic submanifold, and assume either that $(M,\omega)$ is compact and the group $\left\{\left.\int_{S^2}u^*\omega\right|u:S^2\to N\right\}$ is discrete, or else that $(M,\omega)$ is symplectically aspherical, geometrically bounded, and wide.\footnote{``Wide'' means that there is an exhausting Hamiltonian $H\co M\to \R$ having a positive lower bound on the periods of its nontrivial contractible periodic orbits; see \cite{Gu}.}   Suppose moreover that there is a closed subset $S\subset N$ with $e(S,M)=0$ such that every leaf of the characteristic foliation passing through $N\setminus S$ is dense.  Then $N$ is CH-rigid. 
\end{cor}

\begin{proof} If $N$ were not CH-rigid, then by the previous corollary $\hat{N}$ would be weightless.  Now $\hat{N}$ is obviously displaceable (by translations in the $\R^2$ factor), so if $\hat{N}$ were weightless then $\hat{N}$ would have zero displacement energy by Proposition \ref{obvious}. But as noted earlier $\hat{N}$ is stable, and using \cite[Theorem 8.4]{U} in the compact case\footnote{The fact that $\R^2\times M$ is not compact does not pose a problem here, since the support of a Hamiltonian displacing $\hat{N}$ can be embedded in a compact symplectic manifold, as in the proof of \cite[Corollary 8.6]{U}.}, or \cite[Theorem 2.7(i)]{Gi} in the aspherical case one can show that $e(\hat{N},\R^2\times M)>0$, a contradiction.
\end{proof}

\begin{ex}  Let $M$ denote the $6$-dimensional torus $\{(x_1,y_1,x_2,y_2,x_3,y_3)|x_i,y_i\in\mathbb{R}/\mathbb{Z}\}$, and where $\ep,\delta\in \R$ have the property that $1,\ep,\delta$ are linearly independent over $\mathbb{Q}$, endow $M$ with the irrational symplectic form \[ \omega=dx_1\wedge dy_1+dx_2\wedge dy_2+dx_3\wedge dy_3+dy_1\wedge(\ep dx_2+\delta dy_2).\]  Let \[ N=\{x_1+x_2=x_3=0\}\] (where of course the equalities are mod $\mathbb{Z}$), so that $N$ is a coisotropic $4$-torus in $N$.  ($M$ splits as a product of an irrational $4$-torus $M_0$ spanned by $x_1,y_1,x_2,y_2$ and a standard $2$-torus $T$, and $N$ is the product of a hypersurface $N_0\subset M_0$ and a standard meridian $\mu\subset T$.)  Using the $1$-forms $\alpha_1=dy_1$ and $\alpha_2=dy_3$, one sees that $N$ is stable.  The distribution $TN^{\omega}$ may be computed to be spanned by the vectors \[ \delta(\partial_{x_1}-\partial_{x_2})+\partial_{y_1}+(1+\ep)\partial_{y_2}\quad \mbox{and}\quad \partial_{y_3},\] and so by the assumption on $\ep$ and $\delta$ all of the characteristic leaves of $N$ are dense: they are products of dense lines in the $3$-torus $N_0$ with the meridian $\mu$.  Thus Corollary \ref{posdisp} applies to show that $N$ is CH-rigid.
\end{ex}

\section{Generic weightlessness} \label{genwt}

We now begin the proof of Theorem \ref{int1}(ii).  This will involve an iterative use of Lemma \ref{rnp}: at the $r$th step we will show that, for a generic closed submanifold $N$ of codimension at least two, at all points of the rigid locus $R_N$ certain identities must be satisfied by the derivatives up to order $r$ of the embedding of $N$.  Since we obtain new identities for every value of $r$, if the identities are cut out transversely (as one expects to occur generically by the jet transversality theorem) then for a sufficiently large value of $r$ this will prove that $R_N$ is empty and hence that $N$ is weightless.  Before setting up the argument, we will develop some of the algebra underlying these identities, and show that their solution spaces are submanifolds.  

\subsection{Some multilinear algebra}\label{algebra}

Fix throughout this subsection two finite-dimensional real vector spaces $V$ and $W$ and an antisymmetric nondegenerate bilinear form $\omega\co W\times W\to \R$.  
For $k\geq 1$ let $Sym^k(V,W)$ denote the vector space of symmetric, $k$-linear maps $A\co V^k\to W$, and $Mult^k(V,W)$ the vector space of $k$-linear (not necessarily symmetric) maps $A\co V^k\to W$.  
 Define, for any integer $s\geq 2$, a map \begin{align*} \prod_{k=1}^{s} Sym^k(V,W)&\to Mult^{s+1}(V,W) \\ (A_1,\ldots,A_s)&\mapsto \tau_{A_1,\ldots,A_s}\end{align*}  where $\tau_{A_1,\ldots,A_s}$ is given by the formula \begin{align} \tau_{A_1,\ldots,A_s}&(v_0,v_1,\ldots,v_{s-1},v_{s})= \nonumber \\& \sum_{\sigma\in S_{s-1}}\sum_{k=0}^{s-1}\frac{1}{k!(s-k-1)!}\omega\left(A_{k+1}(v_0,v_{\sigma(1)},\ldots,v_{\sigma(k)}),A_{s-k}(v_{\sigma(k+1)},\ldots,v_{\sigma(s-1)},v_{s})\right).\label{tausym}\end{align} Here $S_{s-1}$ denotes as usual the group of permutations of the set $\{1,\ldots,s-1\}$.  Equivalently, in light of the symmetry of the $A_j$, \begin{equation}\label{tauset} \tau_{A_1,\ldots,A_{s}}(v_0,\ldots,v_s)=\sum_{k=0}^{s-1}\sum_{\substack{\scriptscriptstyle{\{1,\ldots,s-1\}}=\\ \scriptscriptstyle{ \{i_1,\ldots,i_{k}\}\coprod\{j_1,\ldots,j_{s-k-1}\}}}}\omega\left(A_{k+1}(v_0,v_{i_1},\ldots,v_{i_k}),A_{s-k}(v_{j_1},\ldots,v_{j_{s-k-1}},v_{s})\right).\end{equation}  Here and below we take it as understood that the partitions $\{1,\ldots,s-1\}=\{i_1,\ldots,i_k\}\coprod\{j_1,\ldots,j_{s-k-1}\}$ appearing in the sum have $i_1<\cdots<i_k$ and $j_1<\cdots<j_{s-k-1}$.
 
Here are some salient properties of the $\tau_{A_1,\ldots,A_{s}}$:

\begin{lemma}\label{tauprop}  We have, for any $A_k\in Sym^k(V,W)$ and $v_k\in V$:\begin{itemize}
\item[(i)] \[ \tau_{A_1,\ldots,A_s}(v_0,v_1,\ldots,v_{s-1},v_s)=-\tau_{A_1,\ldots,A_s}(v_s,v_1,\ldots,v_{s-1},v_0).\]
\item[(ii)] For any $\sigma\in S_{s-1}$, \[ \tau_{A_1,\ldots,A_s}(v_0,v_1,\ldots,v_{s-1},v_s)=\tau_{A_1,\ldots,A_s}(v_0,v_{\sigma(1)},\ldots,v_{\sigma(s-1)},v_s).\]
\item[(iii)]\[ \tau_{A_1,\ldots,A_s}(v_0,v_1,v_2,\ldots,v_{s-1},v_s)+\tau_{A_1,\ldots,A_s}(v_s,v_0,v_2,\ldots,v_{s-1},v_1)+\tau_{A_1,\ldots,A_s}(v_1,v_s,v_2,\ldots,v_{s-1},v_0)=0.\]
\item[(iv)]  For fixed $A_1,\ldots,A_{s-1}$ such that $A_1\co V\to W$ is injective, the map \[ A_s\mapsto \tau_{A_1,\ldots,A_s}\] is an affine surjection from $Sym^s(V,W)$ to the space \[ \mathcal{T}_{s}(V,W)=\{\tau\in Mult^{s+1}(V,W)|\tau\mbox{ obeys (i),(ii),(iii) above}\}.\]
\item[(v)] Where for $j=2,\ldots,s$ we denote by $A_j(v_s,\cdot)$ the element of $Sym^{j-1}(V,W)$ given by including $v_s$ into $A_j$ as the first argument, we have  \[ \left.\frac{d}{dt}\right|_{t=0}\left(\tau_{A_1+tA_2(v_s,\cdot),A_2+tA_3(v_s,\cdot),\ldots,A_{s-1}+tA_{s}(v_s,\cdot)}(v_0,\ldots,v_{s-2},v_{s-1})\right)=\tau_{A_1,\ldots,A_s}(v_0,\ldots,v_{s-2},v_s,v_{s-1}).\] 
\end{itemize}
\end{lemma} 

\begin{proof} For $k=0,\ldots,s-1$ write \[ \tau^{(k)}(v_0,\ldots,v_s)=\frac{1}{k!(s-1-k)!}\sum_{\sigma\in S_{s-1}}\omega\left(A_{k+1}(v_0,v_{\sigma(1)},\ldots,v_{\sigma(k)}),A_{s-k}(v_{\sigma(k+1)},\ldots,v_{\sigma(s-1)},v_{s})\right),\] so that
\[ \tau_{A_1,\ldots,A_s}=\sum_{k=0}^{s-1}\tau^{(k)}.\]  First we show that for each $k$, $\tau^{(k)}+\tau^{(s-1-k)}$ obeys conditions (i), (ii), and (iii), which will obviously show the same for $\tau_{A_1,\ldots,A_s}$.  

That condition (ii) (symmetry in the arguments $v_1,\ldots,v_{s-1}$) holds for each individual $\tau^{(k)}$ is immediate from the definition.  That condition (i) (antisymmetry in the arguments $v_0$ and $v_s$) holds for $\tau^{(k)}+\tau^{(s-1-k)}$ follows quickly from the antisymmetry of $\omega$ and the symmetry of the $A_j$: each term \[ \omega\left(A_{k+1}(v_0,v_{i_1},\ldots,v_{i_k}),A_{s-k}(v_{j_1},\ldots,v_{j_{s-k-1}},v_s)\right) \] that appears in the sum defining $\tau^{(k)}(v_0,v_1,\ldots,v_{s-1},v_s)$ has a corresponding term \[ \omega\left(A_{s-k}(v_s,v_{j_1},\ldots,v_{j_{s-k-1}}),A_{k+1}(v_{i_1},\ldots,v_{i_k},v_0)\right) \] that appears in the sum defining $\tau^{(s-1-k)}(v_s,v_1,\ldots,v_{s-1},v_0)$ (and vice versa), and these terms are opposite to each other since $\omega$ is antisymmetric while $A_{k+1},A_{s-k}$ are symmetric.

We now show that $\tau^{(k)}+\tau^{(s-k-1)}$ obeys property (iii) (concerning the effect of cyclically permuting the arguments $v_0,v_1,v_s$ while leaving the others fixed).  We see that
 \begin{align} \nonumber (&\tau^{(k)}+\tau^{(s-k-1)})(v_0,v_1,v_2,\ldots,v_{s-1},v_{s})=\\ \nonumber  &\sum_{\substack{ \scriptscriptstyle{\{1,\ldots,s-1\}=}\\ \scriptscriptstyle{ \{i_1,\ldots,i_{k}\}\coprod\{j_1,\ldots,j_{s-k-1}\}}}}\scriptstyle{\left(\omega(A_{k+1}(v_0,v_{i_1},\ldots,v_{i_k}),A_{s-k}(v_{j_1},\ldots,v_{j_{s-k-1}},v_s))+\omega(A_{s-k}(v_0,v_{j_1},\ldots,v_{j_{s-k-1}}),A_{k+1}(v_{i_1},\ldots,v_{i_k},v_s))\right)}\\&  \nonumber =\sum_{\substack{\scriptscriptstyle{\{2,\ldots,s-1\}=}\\ \scriptscriptstyle{ \{i_1,\ldots,i_{k}\}\coprod}\\ \scriptscriptstyle{ \{m_1,\ldots,m_{s-k-2}\}}}}\scriptstyle{\left(\omega(A_{k+1}(v_0,v_{i_1},\ldots,v_{i_k}),A_{s-k}(v_1,v_{m_1},\ldots,v_{m_{s-k-2}},v_s))+\omega(A_{s-k}(v_{0},v_1,v_{m_1},\ldots,v_{m_{s-k-2}}),A_{k+1}(v_{i_1},\ldots,v_{i_k},v_s))\right)} \\ \label{longeq}  & \quad +\sum_{\substack{ \scriptscriptstyle{\{2,\ldots,s-1\}=}\\ \scriptscriptstyle{ \{n_1,\ldots,n_{k-1}\}\coprod }\\ \scriptscriptstyle{ \{j_1,\ldots,j_{s-k-1}\}}}}\scriptstyle{\left(
\omega(A_{k+1}(v_0,v_1,v_{n_1},\ldots,v_{n_{k-1}}),A_{s-k}(v_{j_1},\ldots,v_{j_{s-k-1}},v_s))+\omega(A_{s-k}(v_{0},v_{j_1},\ldots,v_{j_{s-k-2}}),A_{k+1}(v_1,v_{n_1},\ldots,v_{n_{k-1}},v_s))\right)}\end{align}
where we have separated the partitions $\{1,\ldots,s-1\}=\{i_1,\ldots,i_k\}\coprod\{j_1,\ldots,j_{s-k-1}\}$ according to whether $1$ belongs to the first or the second of the two subsets.  Fix a partition $\{2,\ldots,s-1\}=\{i_1,\ldots,i_k\}\coprod\{m_1,\ldots,m_{s-k-2}\}$ and consider the effect of cyclically permuting $v_0,v_1,v_s$ in the term corresponding to this partition in the first line of the right hand side of (\ref{longeq}).  Summing over these cyclic permutations yields (where $\ldots$ represents $i_1,\ldots,i_k$ or $m_1,\ldots,m_{s-k-2}$, as appropriate, and where we freely use the symmetry of the $A_j$) \begin{align*} & \omega\left(A_{k+1}(v_0,\ldots),A_{s-k}(v_1,v_s,\ldots)\right)  +\omega\left(A_{s-k}(v_0,v_1,\ldots),A_{k+1}(v_s,\ldots)\right)\\  +&\omega\left(A_{k+1}(v_s,\ldots),A_{s-k}(v_0,v_1,\ldots)\right)+\omega\left(A_{s-k}(v_s,v_0,\ldots),A_{k+1}(v_1,\ldots)\right)\\ +&\omega\left(A_{k+1}(v_1,\ldots),A_{s-k}(v_s,v_0,\ldots)\right)+\omega\left(A_{s-k}(v_1,v_s,\ldots),A_{k+1}(v_0,\ldots)\right),\end{align*} which vanishes, as the first and sixth; second and third; and fourth and fifth terms cancel.  This shows that the terms coming from the first line of the right hand side of (\ref{longeq}) in \[ (\tau^{(k)}+\tau^{(s-k-1)})(v_0,v_1,\ldots,v_s)+(\tau^{(k)}+\tau^{(s-k-1)})(v_1,v_s,\ldots,v_0)+(\tau^{(k)}+\tau^{(s-k-1)})(v_s,v_0,\ldots,v_1)\] sum to zero, and an identical argument applies to the second line of (\ref{longeq}).  This proves property (iii), both for $\tau^{(k)}+\tau^{(s-k-1)}$ and for the original $\tau_{A_1,\ldots,A_k}$.

We now prove (iv).  Since the only terms in $\tau_{A_1,\ldots,A_{s}}=\sum_{k=0}^{s-1}\tau^{(k)}$ which depend on $A_s$ are those corresponding to $k=0,s-1$, and since the terms corresponding to $k=0,s-1$ depend linearly on $A_s$ (with $A_1$ fixed), it suffices to show that, for fixed injective $A_1$, the map $A_s\mapsto \tau^{(0)}+\tau^{(s-1)}$ is a surjection to $\mathcal{T}_s(V,W)$.  (That this map takes values in $\mathcal{T}_s(V,W)$ follows from what we have already done in this proof.)  Note that \begin{equation}\label{0s} (\tau^{(0)}+\tau^{(s-1)})(v_0,v_1,\ldots,v_s)=\omega\left(A_1v_0,A_{s}(v_1,\ldots,v_s)\right)+\omega\left(A_s(v_0,\ldots,v_{s-1}),A_1v_s\right).\end{equation}

Let $\{e_1,\ldots,e_n\}$ be a basis for $V$, and consider an arbitrary $\tau\in \mathcal{T}_s(V,W)$.  Of course $\tau$ is determined by its values on tuples $e_{i_{\sigma(1)}},\ldots,e_{i_{\sigma(s+1)}}$, where  $(i_1,\ldots,i_{s+1})$ varies over tuples with $i_1\leq i_2\leq\cdots\leq i_{s+1}$, and where $\sigma$ varies over $S_{s+1}$.  In fact, though, for a fixed such tuple $(i_1,\ldots,i_{s+1})$, all of the $\tau(e_{i_{\sigma(1)}},\ldots,e_{i_{\sigma(s+1)}})$ are determined by the values \begin{equation}\label{1j} \tau(e_{i_1},e_{i_2},\ldots,\widehat{e_{i_j}},\ldots,e_{i_{s+1}},e_{i_j})\end{equation} where $2\leq j\leq s+1$ varies through indices such that $i_1\neq i_j$ and the $\,\widehat{}\,$ denotes omission.  Indeed if $i_1=i_j$ then (\ref{1j}) vanishes by condition (i), while more generally repeated application of properties (i)-(iii) will express $\tau(e_{i_{\sigma(1)}},\ldots,e_{i_{\sigma(s+1)}})$ in terms of expressions of the form (\ref{1j}) for appropriate $j$; for instance one has \begin{align*} \tau(e_{i_j},e_{i_1},\ldots,\widehat{e_{i_j}},\ldots,\widehat{e_{i_k}},\ldots,e_{i_k})=\tau(e_{i_1}&,\ldots,e_{i_j},\ldots,\widehat{e_{i_k}},\ldots,e_{i_{s+1}},e_{i_k})\\&-\tau(e_{i_1},\ldots,\widehat{e_{i_j}},\ldots,e_{i_k},\ldots,e_{i_{s+1}},e_{i_j}).\end{align*}
So it suffices to show that for a fixed injective $A_1$ and for a fixed $(i_1,\ldots,i_{s+1})$ with $i_1\leq\cdots\leq i_{s+1}$ and for any $j$ with $i_j\neq i_1$, the $s$-linear map $A_s$ may be chosen so that $(\tau^{(0)}+\tau^{(s-1)})(e_{i_1},\ldots,\widehat{e_{i_k}},\ldots,e_{i_{s+1}},e_{i_k})$ is nonzero iff $i_j=i_k$, and so that $(\tau^{(0)}+\tau^{(s-1)})(e_{m_1},\ldots,e_{m_{s+1}})=0$ whenever the indices $m_k$ cannot be reordered to coincide with the indices $i_k$ (since $\mathcal{T}_s(V,W)$ is spanned by maps having these properties as $(i_1,\ldots,i_{s+1})$ and $j$ vary).  But this is easily accomplished.  Choose the symmetric $s$-linear map $A_s\co V^s\to W$ so that $A_s(e_{m_1},\ldots,e_{m_s})=0$ if and only if  $(e_{m_1},\ldots,e_{m_s})$ is not a reordering of $(e_{i_1},\ldots,\widehat{e_{i_{j}}},\ldots,e_{i_{s+1}})$, and so that $A_s(e_{i_1},\ldots,\widehat{e_{i_{j}}},\ldots,e_{i_{s+1}})$ is $\omega$-orthogonal to $A_1e_m$ for all $m\neq i_j$ but is not $\omega$-orthogonal to $A_1e_{i_j}$ (here of course we use the nondegeneracy of $\omega$ and the injectivity of $A_1$).  As the reader may easily verify using (\ref{0s}),  this choice of $A_s$ results in $\tau^{(0)}+\tau^{(s-1)}$ satisfying the desired properties, completing the proof of (iv).

Finally, consider (v). The left-hand side is equal to:
\begin{align*} \sum_{m=1}^{s-1} &\left.\frac{d}{dt}\right|_{t=0}\tau_{A_{1},\ldots,A_{m-1},A_m+tA_{m+1}(v_s,\cdot),A_{m+1},\ldots,A_{s-1}}(v_0,\ldots,v_{s-1})\\&=\sum_{m=1}^{s-1}\sum_{\substack{\scriptscriptstyle{\{1,\ldots,s-2\}=}\\ \scriptscriptstyle{ \{i_1,\ldots,i_{m-1}\}\coprod\{j_1,\ldots,j_{s-m-1}\}}}}\Big{(}\omega\left(A_{m+1}(v_s,v_0,v_{i_1},\ldots,v_{i_{m-1}}),A_{s-m}(v_{j_1},\ldots,v_{j_{s-m-1}},v_{s-1})\right)+ \\& \qquad \qquad \qquad  \qquad \qquad \qquad \qquad  \qquad\omega\left(A_{s-m}(v_0,v_{j_1},\ldots,v_{j_{s-m-1}}),A_{m+1}(v_s,v_{i_1},\ldots,v_{i_{m-1}},v_{s-1})  \right)\Big{)}.\end{align*}  But, in view of the symmetry of the $A_j$, this is just equal to $\tau_{A_1,\ldots,A_s}(v_0,\ldots,v_{s-2},v_{s},v_{s-1})$, as can be seen by sorting the terms that appear in (\ref{tauset}) according to whether the second-to-last argument of $\tau_{A_1,\ldots,A_s}$ is among the $v_{i_k}$ or among the $v_{j_k}$.
\end{proof}

\begin{lemma}\label{keyproj} Where $\mathcal{T}_s(V,W)$ is as defined in Lemma \ref{tauprop}(iv), there exists a surjective linear projection $\Pi\co Mult^{s+1}(V,W)\to \mathcal{T}_s(V,W)$ having the following property.  If $\eta\in \mathcal{T}_s(V,W)$ and $S\leq V$ are such that $\eta(v_0,\ldots,v_s)=0$ for all $v_0,\ldots,v_s\in S$, then it also holds that $(\Pi\eta)(v_0,\ldots,v_s)=0$ for all $v_0,\ldots,v_s\in S$.
\end{lemma}

\begin{proof}
Let $\mathfrak{S}_{s+1}$ denote the symmetric group on the $(s+1)$-element set $\{0,\ldots,s\}$.  For $0\leq i\leq s-1$ let $t_i\in \mathfrak{S}_{s+1}$ denote the transposition which interchanges $i$ and $i+1$.  Also let $t_s$ denote the transposition which interchanges $0$ and $s$, and let $u$ denote the permutation which cyclically permutes $0,1,$ and $s$ and leaves the other elements of $\{0,\ldots,s\}$ unchanged.

We have a left action of $\mathfrak{S}_{s+1}$ on $Mult^{s+1}(V,W)$ by \[ \sigma\cdot\tau(v_0,\ldots,v_s)=\tau(v_{\sigma^{-1}(0)},\ldots,v_{\sigma^{-1}(s)}).\]  Where $\R\mathfrak{S}_{s+1}$ denotes the group $\R$-algebra of $\mathfrak{S}_{s+1}$, the action of $\mathfrak{S}_{s+1}$ on $Mult^{s+1}(V,W)$ extends in the obvious way to a left action of the algebra $\R\mathfrak{S}_{s+1}$ on $Mult^{s+1}(V,W)$.  In these terms we have by definition \[ \mathcal{T}_s(V,W)=\left\{\tau\in Mult^{s+1}(V,W)  \left|\begin{array}{c} (1-t_i)\cdot\tau=0 \mbox{ for }1\leq i\leq s-2   \\ (1+t_s)\cdot\tau=(1+u+u^2)\cdot\tau=0 \end{array} \right.\right\}
\]

Let $I$ denote the left ideal in $\R\mathfrak{S}_{s+1}$ generated by the elements $1-t_i$ for $1\leq i\leq s-2$, $1+t_s$, and $1+u+u^2$. (In other words, $I$ consists of elements of the form $\sum_{i=1}^{s-2}x_i(1-t_i)+y(1+t_s)+z(1+u+u^2)$.)   We then evidently have \begin{equation}\label{tid} \mathcal{T}_{s}(V,W)=\{\tau\in Mult^{s+1}(V,W)|(\forall x\in I)(x\cdot\tau=0)\}.\end{equation}

Now $\R\mathfrak{S}_{s+1}$ carries a $\mathfrak{S}_{s+1}$-invariant inner product $\langle\cdot,\cdot\rangle$ defined by \[\left\langle \sum_{\sigma\in\mathfrak{S}_{s+1}}a_{\sigma}\sigma,\sum_{\sigma\in\mathfrak{S}_{s+1}}b_{\sigma}\sigma\right\rangle=\sum_{\sigma\in\mathfrak{S}_{s+1}}a_{\sigma}b_{\sigma}. \]  Let $I^{\perp}$ denote the orthogonal complement of $I$ with respect to this inner product.  The facts that $\langle\cdot,\cdot\rangle$ is $\mathfrak{S}_{s+1}$-invariant and that $I$ is a left ideal readily imply that $I^{\perp}$ is a left ideal.  Of course there is a direct sum splitting of vector spaces $\R\mathfrak{S}_{s+1}=I\oplus I^{\perp}$, so where $1$ is the multiplicative identity in $\R\mathfrak{S}_{s+1}$ we may write $1=e+e^{\perp}$ where $e\in I$ and $e^{\perp}\in I^{\perp}$.  Now for any $x\in \R\mathfrak{S}_{s+1}$ we have \[ x=x(e+e^{\perp})=xe+xe^{\perp} \] where $xe\in I$ and $xe^{\perp}\in I^{\perp}$.  So if $x\in I$ then $xe=x$ and $xe^{\perp}=0$, while if $x\in I^{\perp}$ then $xe=0$ and $xe^{\perp}=x$.  In particular applying this with $x$ equal to $e$ or $e^{\perp}$ shows that $e^2=e$, $(e^{\perp})^2=e^{\perp}$, and $ee^{\perp}=e^{\perp}e=0$.

If $\tau\in\mathcal{T}_{s}(V,W)$ we have $e\cdot\tau=0$ and hence $e^{\perp}\cdot\tau=(e+e^{\perp})\cdot\tau=\tau$.  

If $\tau\in Mult^{s+1}(V,W)$, and $x\in I$, since $x=xe$ and $ee^{\perp}=0$ we have \[ x(e^{\perp}\cdot\tau)=(xee^{\perp})\cdot\tau=0.\]
Thus by (\ref{tid}) we have $e^{\perp}\cdot \tau\in \mathcal{T}_s(V,W)$ for all $\tau\in Mult^{s+1}(V,W)$.

Accordingly we may define $\Pi\co Mult^{s+1}(V,W)\to \mathcal{T}_s(V,W)$ by $\Pi(\tau)=e^{\perp}\cdot\tau$.  The last two paragraphs together with the fact that  $(e^{\perp})^2=e^{\perp}$ imply that $\Pi$ is a surjective projection.  The fact that $\Pi$ is given by the action of an element of $\R\mathfrak{S}_{s+1}$ immediately implies that $\Pi$ has the property stated in the lemma: if $\eta\in Mult^{s+1}(V,W)$ vanishes on all tuples consisting of elements of the subspace $S\leq V$, then for $v_0,\ldots,v_s\in S$, $(\Pi\eta)(v_0,\ldots,v_s)$ is a linear combination of various terms obtained by first permuting the $v_i$ and then applying $\eta$, and all of these terms are $0$ by the assumption on $\eta$.
\end{proof}

If $A\co V\to W$ is a linear map, we obtain a skew-symmetric bilinear form $A^*\omega$ on $V$.  Associated to the form $\omega$ on $W$ is a linear map $J\co W\to W^*$ defined by the property that $(Jw_1)(w_2)=\omega(w_1,w_2)$.  The linear map $V\to V^*$ which is similarly associated to the bilinear form $A^*\omega$ on $V$ is then $A^*JA$, where $A^*\co W^*\to V^*$ is the transpose of $A$.  Since we assume that $\omega$ is nondegenerate, $J$ is invertible.  On the other hand $A^*\omega$ is typically degenerate; its kernel (\emph{i.e.} the space of those $v$ such that $(A^*\omega)(v,\cdot)\in V^*$ is zero) is the same as the kernel of the linear map $A^*JA$.  If we assume that $A$ is injective, so that $A^*$ is surjective, then it is easy to see that the kernel of $A^*\omega$ has dimension no larger than $\dim W-\dim V$.

\begin{prop}\label{KC} Let $Hom(V,W)$ be the space of linear maps and $Mon(V,W)$ the space of injective linear maps from $V$ to $W$, and let $c$ be a positive integer.  Then \[ \mathcal{K}_c=\{A\in Mon(V,W)|\dim(\ker(A^*\omega))=c\}\] is a submanifold of $Mon(V,W)$ with codimension equal to $\frac{c(c-1)}{2}$.  Moreover for $A\in \mathcal{K}_c$ the tangent space to $\mathcal{K}_c$ at $A$ is given by \[ T_A\mathcal{K}_c=\{B\in Hom(V,W)|\omega(Av_1,Bv_2)+\omega(Bv_1,Av_2)=0\mbox{ for all }v_1,v_2\in \ker A^*\omega\},\] where we use the fact that $Mon(V,W)$ is an open subset of the vector space $Hom(V,W)$ to identify $T_AMon(V,W)$ canonically with $Hom(V,W)$.
\end{prop}

\begin{proof} If $U$ is a finite-dimensional vector space let $Sk(U)$ denote the vector space of skew-symmetric linear maps $L\co U\to U^*$ (in other words, maps that, under the canonical identification of $U^{**}$ with $U$, obey $L^*=-L$; of course these are the maps that, when represented by matrices in terms of a basis for $U$ and the corresponding dual basis for $U^*$, are given by skew-symmetric matrices).

Let \[ Q_c=\{B\in Sk(V)|\dim\ker B=c\}\] and for 
 any subspace $Y\leq V$, let \[  Sk^Y(V)=\{B\in Sk(V)|(By_1)(y_2)=0\mbox{ for all }y_1,y_2\in Y\}.\]
 
 \begin{lemma}\label{qc}  $Q_c$ is a submanifold of $Sk(V)$, with codimension $\frac{c(c-1)}{2}$.  Moreover the tangent space at $B\in Q_c$ is given by \[ T_{B}Q_c=Sk^{\ker B}(V) \] (where we use the vector space structure on $Sk(V)$ to identify $T_BSk(V)$ with $Sk(V)$).\end{lemma}
 
\begin{proof}  Let $B\in Q_c$ and write $Y^0=\ker B$, so $\dim Y^0=c$.  Choose a complement $Y^1$ to $Y^0$ in $V$.  The splitting $V=Y^0\oplus Y^1$ determines a splitting $V^*=(Y^{0})^{*}\oplus (Y^{1})^{*}$ (where, \emph{e.g.}, elements of $(Y^{0})^{*}$ are extended by zero on the $Y^1$ summand to obtain elements of $V^*$), and any $C\in Sk(V)$ can then be written in ``block form'' as $C=\left(\begin{array}{cc} C_{00} & C_{01} \\ C_{10} & C_{11}\end{array}\right)$ where $C_{ij}\co Y^j\to (Y^i)^*$ and where $C_{ij}^{*}=-C_{ji}$.
The fact that $B$ vanishes  on $Y^0$ shows that $B_{00}=0$ and $B_{10}=0$, and hence also by skew-symmetry $B_{01}=0$.  So since $B$ is injective on $Y^1$ the lower right block $B_{11}$ must be invertible.  Let $U$ denote the open subset of $Sk(V)$ consisting of skew-symmetric maps $B'$ whose lower right blocks $B'_{11}$ are invertible, so $U$ is an open neighborhood of $B$ and it suffices to show that $Q_{c}\cap U$ is a submanifold of $U$ with codimension and tangent space at $B$ as asserted in the statement of the lemma.    

Let $\pi_0$ be the projection $V\to Y^0$ given by the direct sum splitting $V=Y^0\oplus Y^1$.
Now any $B'\in U$ restricts injectively to $Y^1$ and so the restriction of $\pi_0$ to $\ker B'$ is injective. Assuming that $B'\in U$, we have $B'\in Q_c\cap U$ if and only if $\pi_0|_{\ker B'}$ is an isomorphism to $Y^0$, which in turn holds if and only if there is a linear map $D\co Y^0\to Y^1$ so that $B'$ vanishes identically on the subspace $\{v+Dv|v\in Y^0\}$.  Writing \[ B'=\left(\begin{array}{cc} C_{00} & C_{01} \\ C_{10} & B_{11}+C_{11}\end{array}\right),\] the precise conditions on the $C_{ij}$ and on $D$ for this to occur are given by \begin{align*}  C_{00}+C_{01}D&=0\\ C_{10}+(B_{11}+C_{11})D&=0 \end{align*}  Since $B'$ is chosen from the open set $U$, the map $B_{11}+C_{11}$ is invertible, and so $D$ would have to be given by $D=-(B_{11}+C_{11})^{-1}C_{10}$.  So since $C_{01}=-C_{10}^{*}$ we see that \[ Q_c\cap U=\left\{\left.B'=\left(\begin{array}{cc} C_{00} & -C_{10}^{*} \\ C_{10} & B_{11}+C_{11}\end{array}\right)\right| \begin{array}{c}C_{00}\in Sk(Y^0),C_{11}\in Sk(Y^1),C_{10}\in Hom(Y^0,(Y^{1})^{*}),\\ C_{00}+C_{10}^{*}(B_{11}+C_{11})^{-1}C_{10}=0\end{array}\right\}.\]

The map \begin{equation} \label{sksub}
(C_{00},C_{10},C_{11})\mapsto C_{00}+C_{10}^{*}(B_{11}+C_{11})^{-1}C_{10} \end{equation} is obviously a submersion to $Sk(Y^0)$, and $Sk(Y^0)$ has dimension $\frac{c(c-1)}{2}$, so by the implicit function theorem this proves that $Q_c$ is a submanifold of $Sk(V)$ with codimension $\frac{c(c-1)}{2}$.  Moreover, the linearization of (\ref{sksub}) around $C_{ij}=0$ has kernel given precisely by those $\hat{C}$ with $\hat{C}_{00}=0$, and the condition that $\hat{C}_{00}=0$ amounts to the statement that $\hat{C}\in Sk^{Y_0}(V)$, proving that $T_{B}Q_c=Sk^{Y^0}(V)$.
\end{proof}
 
Resuming the proof of Proposition \ref{KC}, first note that the map $\Omega\co Mon(V,W)\to Sk(V)$ defined by $\Omega(A)=A^*JA$ is a submersion.  Indeed the linearization of this map at $A$ is given by $B\mapsto B^*JA+A^*JB=B^*(JA)-(JA)^*B$ where $JA\co V\to W^*$ is a monomorphism (here we use that $J^*=-J$ since $\omega$ is skew-symmetric).   Choosing $C\in Hom(W^*,V)$ so that $C(JA)$ is the identity on $V$, given any $D\in Hom(V,V^*)$ the element $C^*D\in Hom(V,W)$ will be sent by the linearization of $\Omega$ at $A$ to $D^*-D$.  So since any element of $Sk(V)$ can be written as $D^*-D$ for some $D\in Hom(V,V^*)$, $\Omega$ is indeed a submersion.

We now need simply note that we have \[ \mathcal{K}_c=\Omega^{-1}(Q_c),\] so since $Q_c$ is a submanifold of codimension $\frac{c(c-1)}{2}$ and $\Omega$ is a submersion, $\mathcal{K}_c$ is a submanifold of codimension $\frac{c(c-1)}{2}$.  Moreover for $A\in \mathcal{K}_c$ the tangent space to $\mathcal{K}_c$ at $A$ consists of those $B$ such that $B^*JA+A^*JB$ belongs to $T_{A^*JA}Q_c=Sk^{\ker(A^*JA)}(V)$.  Recalling that $\ker(A^*JA)=\ker A^*\omega$, this amounts to the condition that, for all $v_1,v_2\in \ker A^*\omega$, we have \[ 0=\left((B^*JA+A^*JB)(v_1)\right)(v_2)=\left(JAv_1\right)(Bv_2)+\left(JBv_1\right)(Av_2)= \omega(Av_1,Bv_2)+\omega(Bv_1,Av_2),\] as desired.
\end{proof}

\begin{add}\label{add} Fix a $c$-dimensional subspace $V^0\leq V$.  For any $A\in \mathcal{K}_c$ there is a neighborhood $U$ of $A$ in $Mon(V,W)$ and a smooth map \begin{align*} \Psi\co U\times V& \to V \\ (Z,v)&\mapsto \Psi_Z(v) \end{align*} such that for all $Z\in U$ the map $\Psi_Z\co V\to V$ is a linear isomorphism, and such that for all $Z\in U\cap \mathcal{K}_c$ we have $\Psi_Z(V^0)=\ker(Z^*\omega)$.
\end{add}

\begin{proof}
This basically follows from the discussion in the proof of Lemma \ref{qc}.  
  Let $Y^0=\ker A^*JA$, and choose a complement $Y^1$ to $Y^0$ in $V$.  As in the proof of Lemma \ref{qc} we can write $A^*JA$ in block form with respect to the splitting $V=Y^0\oplus Y^1$ as $A^*JA=\left(\begin{array}{cc}0 & 0 \\ 0 & B_{11}\end{array}\right)$ where $B_{11}\co Y^1\to (Y^1)^*$ is invertible.  Let $U$ denote the set of $Z\in Mon(V,W)$ such that the lower right block of $Z^*JZ$ with respect to the splitting $V=Y^0\oplus Y^1$ is invertible. For any $Z\in U$ define maps $C_{ij}(Z)\co Y^j\to (Y^i)^*$ by the property that \[ Z^*JZ=\left(\begin{array}{cc}C_{00}(Z) & C_{01}(Z) \\ C_{10}(Z) & B_{11}+C_{11}(Z)\end{array}\right).\] Then the $C_{ij}(Z)$ vary smoothly with $Z$ and it holds that $C_{ij}(Z)^{*}=-C_{ji}(Z)$, that $C_{ij}(A)=0$, and that $B_{11}+C_{11}(Z)$ is invertible.  As noted in the proof of Lemma \ref{qc}, given that $Z^*JZ$ takes the above form, if $\ker(Z^*JZ)$ is to have dimension $c$, then it must hold that $\ker(Z^*JZ)=\{v+D(Z)v|v\in Y^0\}$, where $D(Z)\co Y^0\to Y^1$ is given by the formula  \begin{equation}\label{dz}D(Z)=-(B_{11}+C_{11}(Z))^{-1}C_{10}(Z).\end{equation}

To construct the desired map $\Psi$, where $V^0\leq V$ is our given $c$-dimensional subspace, choose a complement $V^1$ to $V^0$ in $V$, and for $i=0,1$ fix linear isomorphisms $\psi_i\co V^i\to Y^i$.  Then define $\Psi\co U\times V\to V$ by \[ \Psi(Z,v_0+v_1)=\psi_0v_0+D(Z)\psi_0v_0+\psi_1v_1,\] where $D(Z)\co Y^0\to Y^1$ is given by the formula (\ref{dz}) (of course, this formula makes sense as long as $Z\in U$, whether or not $Z\in\mathcal{K}_c$).  This map is easily seen to satisfy the desired properties.
\end{proof}

For the rest of this subsection we will fix a smooth map $\Psi\co U_{\Psi}\times V\to V$ as in Addendum \ref{add}; thus $U_{\Psi}$ is an open set in $Mon(V,W)$, $V^0$ is a fixed $c$-dimensional subspace of $V$, and the maps $\Psi_Z=\Psi(Z,\cdot)\co V\to V$ are, for each $Z\in U_{\Psi}$, linear isomorphisms such that whenever $Z\in U_{\Psi}\cap \mathcal{K}_c$ we have $\Psi_Z(V^0)=\ker(Z^*JZ)$.

In general, for $\eta\in Mult^{s+1}(V,W)$ and $f\co S\to V$ a linear map from some  vector space $S$, $f^{*}\eta$ denotes the obvious pullback of $\eta$, \emph{i.e.}, $f^{*}\eta\in Mult^{s+1}(V,W)$ is given by $f^{*}\eta(x_0,\ldots,x_{s})=\eta(fx_0,\ldots,fx_s)$.

We will consider smooth maps  \begin{align*} \eta\co U_{\Psi}\times\prod_{k=2}^{s}Sym^k(V,W)&\to Mult^{s+1}(V,W)   
\\ (A_1,\ldots,A_s)&\mapsto \eta_{A_1,\ldots,A_s} 
\end{align*} 

The domain of such a map should be thought of as consisting of possible values of the derivatives at a point of order $1$ through $s$ of a function $f\co V\to W$ (with the first derivative constrained to lie in the open set $U_{\Psi}$ but the higher order derivatives allowed to vary freely).   

We associate to such a $\Psi$ and to any integer $s\geq 2$ a map \[ \mathcal{F}_{\Psi}\co C^{\infty}\left(U_{\Psi}\times\prod_{k=2}^{s}Sym^k(V,W),Mult^{s+1}(V,W)\right)\to C^{\infty}\left(U_{\Psi}\times\prod_{k=2}^{s+1}Sym^k(V,W),Mult^{s+2}(V,W)\right) \] defined by \begin{align*} (\mathcal{F}_{\Psi}\eta)_{A_1,\ldots,A_{s+1}}(v_0&,v_1,\ldots,v_{s-1},v,v_{s})=\\& \left.\frac{d}{dt}\right|_{t=0}\left((\Psi_{A_1+tA_2(v,\cdot)}\circ\Psi_{A_1}^{-1})^{*}\eta_{A_1+tA_2(v,\cdot),A_2+tA_3(v,\cdot),\ldots,A_s+tA_{s+1}(v,\cdot)}(v_0,v_1,\ldots,v_{s-1},v_{s})\right).\end{align*} 

To give some sort of motivation for this, note that if the $A_i$  are the $i$th order derivatives at a point of a function $f\co V\to W$, then $\left.\frac{d}{dt}\right|_{t=0}(A_1+tA_2(v,\cdot),A_2+tA_3(v,\cdot),\ldots,A_s+tA_{s+1}(v,\cdot))$ measures the rate of change of the first $s$ derivatives of $f$ as one moves in the direction $v$.  Thus $(\mathcal{F}_{\Psi}\eta)_{A_1,\ldots,A_{s+1}}(\cdot,v,\cdot)$ is a measurement of the change in $\eta$ for a function $f$ with derivatives $A_i$ as one moves in the direction $v$.  The pullback by $\Psi_{A_1+tA_2(v,\cdot)}\circ\Psi_{A_1}^{-1}$ is designed to compensate for the fact that the subspace $\ker(f^{*}\omega)_x$ will vary as $x\in V$ varies.

By the chain rule we have \begin{equation}\label{df} (\mathcal{F}_{\Psi}\eta)_{A_1,\ldots,A_{s+1}}(\cdot,v,\cdot)=\left.\frac{d}{dt}\right|_{t=0}\left(\eta_{A_1+tA_2(v,\cdot),A_2+tA_3(v,\cdot),\ldots,A_s+tA_{s+1}(v,\cdot)}+(\Psi_{A_1}^{-1})^{*}\Psi_{A_1+tA_2(v,\cdot)}^{*}\eta_{A_1,\ldots,A_s}\right)\end{equation}

Note the similarity of the first term  in (\ref{df}) to what appears in Lemma \ref{tauprop}(v), and also note that the second term is independent of $A_{s+1}$.

Now choose, for all $s\geq 2$, a projection $\Pi\co Mult^{s+1}(V,W)\to \mathcal{T}_{s}(V,W)$ as in Lemma \ref{keyproj}.  Define elements $\tilde{\tau}^{\Psi,s}\in C^{\infty}(U_{\Psi}\times\prod_{k=2}^{s}Sym^k(V,W),\mathcal{T}_s(V,W))$ inductively by setting, where $\tau_{A_1,A_2}$ is as defined before Lemma \ref{tauprop}, \[ \tilde{\tau}^{\Psi,2}_{A_1,A_2}=\tau_{A_1,A_2} \]  and, for $s\geq 2$, \[ \tilde{\tau}^{\Psi,s+1}_{A_1,\ldots,A_{s+1}}=\Pi\left(\mathcal{F}_{\Psi}\tilde{\tau}^{\Psi,s}\right)_{A_1,\ldots,A_{s+1}}.\]

\begin{remark}  Our purpose in including the projection $\Pi$ in the definition of $\tilde{\tau}^{\Psi,s}$ is to ensure that the map $(A_1,\ldots,A_s)\mapsto \tilde{\tau}^{\Psi,s}_{A_1,\ldots,A_{s}}$ has constant rank.
\end{remark}

\begin{lemma}\label{tausymb} For each $s\geq 2$ there is a $C^{\infty}$ map $g_s\co U_{\Psi}\times \prod_{k=2}^{s-1}Sym^k(V,W)\to \mathcal{T}_s(V,W)$ such that, for all $(A_1,A_2,\ldots,A_s)\in U_{\Psi}\times \prod_{k=2}^{s}Sym^k(V,W)$ we have \[ \tilde{\tau}^{\Psi,s}_{A_1,\ldots,A_{s}}=\tau_{A_1,\ldots,A_s}+g_s(A_1,\ldots,A_{s-1}).\]  Consequently for any fixed
$(A_1,A_2,\ldots,A_{s-1})\in U_{\Psi}\times \prod_{k=2}^{s-1}Sym^k(V,W)$ the map $A_{s}\mapsto \tilde{\tau}^{\Psi,s}_{A_1,\ldots,A_{s}}$ is an affine surjection from $Sym^{s}(V,W)$ to $\mathcal{T}_s(V,W)$.
\end{lemma}

\begin{proof}  Given Lemma \ref{tauprop}, this follows easily by induction on $s$. Of course it is trivially true for $s=2$.  Assuming the first statement of the lemma for some $s\geq 2$, note that the maps $\bar{\tau}_{(s)} \co U_{\Psi}\times \prod_{k=1}^{s}Sym^k(V,W)\to Mult^{s+1}(V,W)$ defined by $\bar{\tau}_{(s)}(A_1,\ldots,A_s)=\tau_{A_1,\ldots,A_s}$ take values in $\mathcal{T}_s(V,W)$ by Lemma \ref{tauprop}(i)-(iii) (so $\Pi\tau_{A_1,\ldots,A_{s+1}}=\tau_{A_1,\ldots,A_{s+1}})$, and by Lemma \ref{tauprop}(v) and (\ref{df}) we have $(\mathcal{F}_{\Psi}\bar{\tau}_{(s)})_{A_1,\ldots,A_{s+1}}=\tau_{A_1,\ldots,A_{s+1}}+h_s(A_1,\ldots,A_s)$ for some smooth $h_s\co U_{\Psi}\times \prod_{k=2}^{s}Sym^k(V,W)\to Mult^{s+2}(V,W)$.  Consequently 
\[\tilde{\tau}^{\Psi,s+1}_{A_1,\ldots,A_{s+1}}=\tau_{A_1,\ldots,A_{s+1}}+\Pi h_s(A_1,\ldots,A_s)+\Pi\left(\mathcal{F}_{\Psi}g_s\right)_{A_1,\ldots,A_{s+1}}.\] (Here we are strictly speaking extending the domain of $g_s$ to $U_{\Psi}\times \prod_{k=2}^{s}Sym^k(V,W)$ by having it be independent of its last argument $A_s$.) But from the formula for $\mathcal{F}_{\Psi}$ it is clear that the fact that $g_s$ depends only on $A_1,\ldots,A_{s-1}$ implies that $\mathcal{F}_{\Psi}g_s$ depends only on $A_1,\ldots,A_s$.  So the first statement of the lemma holds for the value $s+1$, with $g_{s+1}(A_1,\ldots,A_s)=\Pi h_s(A_1,\ldots,A_s)+\Pi\left(\mathcal{F}_{\Psi}g_s\right)_{A_1,\ldots,A_{s},0}$.

This proves the first statement of the lemma by induction, and then the second statement follows from Lemma \ref{tauprop}(iv) since all elements of $U_{\Psi}$ are monomorphisms.
\end{proof}

\begin{prop}\label{KCR} For any integer $r\geq 1$ and any $\Psi\co U_{\Psi}\times V\to V$ as in Addendum \ref{add}, let \[ \mathcal{K}_{c}^{r,\Psi}(V,W)=\left\{(A_1,\ldots,A_{r})\in  \prod_{k=1}^{r}Sym^k(V,W)\left|\begin{array}{c}A_1\in U_{\Psi}\cap \mathcal{K}_{c}\mbox{ and for all }2\leq s\leq r,\,\\ \tilde{\tau}^{\Psi,s}_{A_{1},\ldots,A_{s}}(v_0,v_1,\ldots,v_s)=0\mbox{ for all }v_0,\ldots,v_s\in \ker(A_{1}^{*}\omega)\end{array}\right.\right\}.\]  Then $\mathcal{K}_{c}^{r,\Psi}(V,W)$ is a submanifold of $U_{\Psi}\times\prod_{k=2}^{r}Sym^k(V,W)$, with codimension equal to \[ \frac{c(c-1)}{2}+\sum_{s=2}^{r}\dim\mathcal{T}_{s}(\R^c,W)\] 
\end{prop}

\begin{proof}  Where $\iota\co V^0\to V$ is the inclusion, we have \[ \mathcal{K}_{c}^{r,\Psi}(V,W)=\left\{(A_1,\ldots,A_{r})\in  \prod_{k=1}^{r}Sym^k(V,W)\left|\begin{array}{c}A_1\in U_{\Psi}\cap \mathcal{K}_{c}\mbox{ and for all }2\leq s\leq r,\,\\ \iota^*\Psi_{A_1}^{*}\tilde{\tau}^{\Psi,s}_{A_1,\ldots,A_s}=0\in \mathcal{T}_s(V^0,W)\end{array}\right.\right\}.\]  

For each $s\in\{2,\ldots,r\}$, the fact that, by Lemma \ref{tausymb}, for any fixed $A_1,\ldots,A_{s-1}$ the map $A_s\mapsto \tilde{\tau}^{\Psi,s}_{A_1,\ldots,A_s}$ is an affine surjection to $\mathcal{T}_s(V,W)$ implies (since $\Psi_{A_1}\circ\iota$ is injective) that, again for fixed $A_1,\ldots,A_{s-1}$, $A_s\mapsto \iota^*\Psi_{A_1}^{*}\tilde{\tau}^{\Psi,s}_{A_1,\ldots,A_s}$ is an affine surjection (and hence a submersion) to $\mathcal{T}_s(V^0,W)$.  This readily implies that, for any fixed $A_1\in U\cap \mathcal{K}_{c}$, the map \begin{align*} \prod_{s=2}^{r}Sym^s(V,W) & \to \prod_{s=2}^{r}\mathcal{T}_s(V^0,W) \\ (A_2,\ldots,A_r) & \mapsto \left(\iota^*\Psi_{A_1}^{*}\tilde{\tau}^{\Psi,2}_{A_1,A_2},\iota^*\Psi_{A_1}^{*}\tilde{\tau}^{\Psi,3}_{A_1,A_2,A_3},\ldots,\iota^*\Psi_{A_1}^{*}\tilde{\tau}^{\Psi,r}_{A_1,\ldots,A_r}\right) \end{align*} is a submersion. 

Now by Proposition \ref{KC}, $(U_{\Psi}\cap \mathcal{K}_c)\times \prod_{s=2}^{r}Sym^s(V,W)$ is a submanifold  of $U_{\Psi}\times \prod_{s=2}^{r}Sym^s(V,W)$ with codimension $\frac{c(c-1)}{2}$.  It follows from the previous paragraph that $\mathcal{K}_{c}^{r,\Psi}(V,W)$ is the zero locus of a submersion from $(U_{\Psi}\cap \mathcal{K}_c)\times \prod_{s=2}^{r}Sym^s(V,W)$ to the vector space $\prod_{s=2}^{r}\mathcal{T}_s(V^0,W)$.  Thus $\mathcal{K}_{c}^{r,\Psi}(V,W)$ is a submanifold with codimension $\sum_{s=2}^{r}\dim\mathcal{T}_s(V^0,W)$ in $(U_{\Psi}\cap \mathcal{K}_c)\times \prod_{s=2}^{r}Sym^s(V,W)$, and therefore codimension $\frac{c(c-1)}{2}+\sum_{s=2}^{r}\dim\mathcal{T}_s(V^0,W)$ in $U_{\Psi}\times \prod_{s=2}^{r}Sym^s(V,W)$.  Recalling that $V^0$ has dimension $c$, this proves the proposition.
\end{proof}

\subsection{Jets and the rigid locus}

We will now incorporate the foundations built in Section \ref{algebra} into the theory of jet spaces; this will culminate in the proof of Theorem \ref{int1}(ii).  In outline, we will soon define what it means, for any positive integer $r$, for a map to be ``transversely $r$-noncoisotropic,'' first in the context of maps from open sets in $\mathbb{R}^d$ to symplectic Euclidean space (see Definition \ref{euctrn}), and then more generally for maps from any $d$-dimensional manifold into a $2n$-dimensional symplectic manifold (see Definition \ref{mantrn}).    Using Thom's jet transversality theorem together with Proposition \ref{KCR}, we will show that if $d\leq 2n-2$ then the set of transversely $r$-noncoisotropic maps is residual in appropriate topologies (see Lemmas \ref{gentrn0} and \ref{gentrn}), and open in the case of spaces of maps from a compact manifold into a symplectic manifold.  Meanwhile, Proposition \ref{trnwt} will show that, for $r$ greater than a dimensional constant, the image of every transversely $r$-noncoisotropic embedding is weightless.  The key ingredient in the proof of Proposition \ref{trnwt} is Lemma \ref{trn}, which implies that for a transversely $r$-noncoisotropic embedding one can set up the sort of iterative scheme based on Lemma \ref{rnp} that was alluded to in the first paragraph of Section \ref{genwt}.  With Proposition \ref{trnwt} in hand, one quickly obtains Corollary \ref{wtmain} and hence Theorem \ref{int1}(ii).

Let $\mathcal{O}\subset \R^d$ be an open subset, and for $1\leq r\leq \infty$ let $C^r(\mathcal{O},\R^{2n})$ denote the space of $C^r$ maps from $\mathcal{O}$ to $\R^{2n}$, endowed with the strong $C^r$ topology (see \cite[Section 2.1]{Hi}).  For $1\leq r<\infty$ let $J^r(\mathcal{O},\R^{2n})$ denote the manifold of $r$-jets of maps from $\mathcal{O}$ to $\R^{2n}$; since $\mathcal{O}$ is assumed to be an open subset of $\R^d$ (so that $T_x \mathcal{O}$ has a fixed identification with $\R^d$ for all $x\in \mathcal{O}$) we may identify \[ J^r(\mathcal{O},\R^{2n})=\{(x,y,A_1,\ldots,A_r)|x\in \mathcal{O},y\in \R^{2n},A_i\in Sym^i(\R^d,\R^{2n})\}\}.\]

For positive integers $c$ and $r$, and for a map $\Psi\co U_{\Psi}\times \R^d\to \R^d$ as in Addendum \ref{add} (where $U_{\Psi}\subset Mon(\R^d,\R^{2n})$ is an open subset) define \begin{equation}\label{phidef} \Phi_{c}^{r,\Psi}=\{(x,y,A_1,\ldots,A_r)\in J^r(\mathcal{O},\R^{2n})|(A_1,\ldots,A_r)\in \mathcal{K}_{c}^{r,\Psi}(\R^d,\R^{2n})\},\end{equation} where $\mathcal{K}_{c}^{r,\Psi}(\R^d,\R^{2n})$ has been defined in Proposition \ref{KCR} (and we use the standard symplectic form $\sum_{i=1}^{n}dy_i\wedge dy_{n+i}$ on $\R^{2n}$). One sees immediately from Proposition \ref{KCR} that $\Phi_{c}^{r,\Psi}$ is a submanifold of $J^r(\mathcal{O},\R^{2n})$ of codimension $\frac{c(c-1)}{2}+\sum_{s=2}^{r}\dim\mathcal{T}_s(\R^c,\R^{2n})$.

Recall that to any $C^{r+1}$ map $f\co \mathcal{O}\to \R^{2n}$ one may associate the $r$-jet of $f$, which is a $C^1$ map $j^rf\co \mathcal{O}\to J^r(\mathcal{O},\R^{2n})$ defined by $j^rf(x)=(x,f(x),(df)_x,\ldots,(d^rf)_x)$, where $(d^if)_x$ denotes the $i$th derivative of $f$ at $x$, viewed as a symmetric $i$-linear form from $T_xU\cong \R^d$ to $\R^{2n}$.  

\begin{dfn}\label{euctrn} An $C^{r+1}$ map $f\co \mathcal{O}\to\R^{2n}$ will be called \emph{transversely $r$-noncoisotropic} if for all $x\in\mathcal{O}$ such that $(df)_x\in\mathcal{K}_{2n-d}$ there is a map $\Psi\co U_{\Psi}\times \R^d\to \R^d$ as in Addendum \ref{add} such that $(df)_x\in U_{\Psi}$ and such that, for all $s\in \{1,\ldots,r\}$, the $s$-jet $j^sf\co \mathcal{O}\to J^s(\mathcal{O},\R^{2n})$ is transverse to $\Phi_{2n-d}^{s,\Psi}$.
\end{dfn}

\begin{lemma}\label{gentrn0} For any  $a>r$ (where $a\in \mathbb{N}\cup\{\infty\}$), the set of $f\in C^a(\mathcal{O},\R^{2n})$ such that $f$ is transversely $r$-noncoisotropic is residual in the strong $C^a$ topology on $C^a(\mathcal{O},\R^{2n})$.
\end{lemma}

\begin{proof} Choose a countable collection of maps $\Psi_i\co U_{\Psi_i}\times \R^d\to \R^d$ as in Addendum \ref{add} such that the open sets $\{U_{\Psi_i}|i\in\mathbb{N}\}$ cover $\mathcal{K}_{2n-d}$.  If $f\co\mathcal{O}\to \R^d$ has the property that $j^sf\pitchfork \Phi_{2n-d}^{s,\Psi_i}$ for all $i\in \mathbb{N}$ and all $s\in\{1,\ldots,r\}$, then $f$ will be transversely $r$-noncoisotropic.  By Thom's jet transversality theorem (see, \emph{e.g.}, \cite[Theorem 3.2.8]{Hi}), for any given $i$ and $s$ with $s<a$ the set of $f$ such that $j^sf\pitchfork \Phi_{2n-d}^{s,\Psi_i}$ is residual in the strong $C^a$ topology on $C^a(\mathcal{O},\R^{2n})$.  Consequently the set described in the lemma contains a countable intersection of residual subsets and therefore is residual (in the strong $C^a$ topology for $a>r$).
\end{proof}

\begin{lemma}\label{trn} Let $f\in C^{r+1}(\mathcal{O},\R^{2n})$ and suppose that $\Psi\co U_{\Psi}\times \R^d\to \R^d$ as in Addendum \ref{add} has the property that $j^s f\pitchfork \Phi_{2n-d}^{s,\Psi}$ for each $s\in \{1,\ldots,r\}$. Define, for $1\leq s\leq r$, \[ \mathcal{O}_{s,f}^{\Psi}=(j^sf)^{-1}(\Phi_{2n-d}^{s,\Psi}).\]  Then each $\mathcal{O}_{s,f}^{\Psi}$ is a submanifold of $\mathcal{O}$ and we have \[ \mathcal{O}_{1,f}^{\Psi}=\{x\in \mathcal{O}|(df)_x\in U_{\Psi}\mbox{ and }\dim\ker (f^*\omega)_{x}=2n-d\} \] and, for $1\leq s\leq r-1$, \begin{equation}\label{us+1} \mathcal{O}_{s+1,f}^{\Psi}\supset\{x\in \mathcal{O}_{s,f}^{\Psi}|\ker(f^{*}\omega)_x\subset T_{x}\mathcal{O}_{s,f}^{\Psi}\}.\end{equation}
\end{lemma}

\begin{proof}  The implicit function theorem and the transversality assumption of course imply that the $\mathcal{O}_{s,f}^{\Psi}$ are submanifolds of $\mathcal{O}$.  Since $j^1f(x)=(x,f(x),(df)_x)$, the statement about $\mathcal{O}_{1,f}^{\Psi}$ follows immediately from the definition of $\Phi_{2n-d}^{1,\Psi}=\mathcal{O}\times \R^{2n}\times (U_{\Psi}\cap \mathcal{K}_{2n-d})$.  

Now let $1\leq s\leq r-1$ and consider $\mathcal{O}_{s+1,f}^{\Psi}$.  First note that under the obvious projection $\pi\co J^{s+1}(\mathcal{O},\R^{2n})\to J^s(\mathcal{O},\R^{2n})$ we have \[ \Phi_{2n-d}^{s+1,\Psi}\subset \pi^{-1}(\Phi_{2n-d}^{s,\Psi})\quad \mbox{ and }\quad \pi\circ j^{s+1}f=j^sf,\] so $\mathcal{O}_{s+1,f}^{\Psi}\subset \mathcal{O}_{s,f}^{\Psi}$.  Now the linearization of $j^s f$ is given by, for $x\in \mathcal{O}$ and $v\in T_x \mathcal{O}$, \[ (j^sf)_{*x}v=\left(v,f_*v,(d^2f)_x(v,\cdot),\ldots,(d^{s+1}f)_x(v,\cdot)\right)\in \R^d\times \R^{2n}\times\prod_{i=1}^{s}Sym^i(\R^d,\R^{2n})\cong T_{j^sf(x)}J^s(\mathcal{O},\R^{2n}).\] 

Let $x\in\mathcal{O}_{s,f}^{\Psi}$.  Then an element $v\in T_x\mathcal{O}$ belongs to $T_{x}\mathcal{O}_{s,f}^{\Psi}$ if and only if $(j^s f)_{*}v\in T_{j^s f(x)}\Phi_{2n-d}^{s,\Psi}$.  By the characterization of $\mathcal{K}_{c}^{r,\Psi}(\R^{d},\R^{2n})$ at the start of the proof of Proposition \ref{KCR} and the definition (\ref{phidef}) of $\Phi_{2n-d}^{r,\Psi}$ this holds if and only if $v\in T_{x}\mathcal{O}_{1,f}^{\Psi}$ and, for each $2\leq m\leq s$, it holds that \[ \left.\frac{d}{dt}\right|_{t=0}\iota^*\Psi_{(df)_x+t(d^2f)_x(v,\cdot)}^{*}\tilde{\tau}^{\Psi,m}_{(df)_x+t(d^2f)_x(v,\cdot) ,\ldots,(d^m f)_x+t(d^{m+1}f)_x(v,\cdot)}=0.\]

Now by definition the left hand side immediately above is precisely the $(m+1)$-linear form on $V^0$ given by \[ (z_0,\ldots,z_{m})\mapsto (\mathcal{F}_{\Psi}\tilde{\tau}^{\Psi,m})_{(df)_x,(d^2f)_x,\ldots,(d^{m+1}f)_x}(\Psi_{(df)_x}z_0,\ldots,\Psi_{(df)_x}z_{m-1},v,\Psi_{(df)_x}z_m)\] (note that here $3\leq m+1\leq s+1$, and recall that $\Psi_{(df)_x}$ maps the model $(2n-d)$-dimensional subspace $V^0\leq \R^d$ isomorphically to $\ker((df)_{x}^{*}\omega)$).   Meanwhile in view of Proposition \ref{KC}, $v\in T_{x}\mathcal{O}_{1,f}^{\Psi}$ if and only if $\omega((df)_xv_1,(d^2f)_x(v,v_2))+\omega((d^2f)_x(v,v_1),v_2)=0$, \emph{i.e.}, $\tilde{\tau}^{\Psi,2}_{(df)_x,(d^2f)_x}(v_1,v,v_2)=0$, for all $v_1,v_2\in V$.  

In view of this, we have \begin{align}\nonumber \{x\in & \mathcal{O}_{s,f}^{\Psi}|\ker(f^*\omega)_x\subset T_{x}\mathcal{O}_{s,f}^{\Psi}\}  \\ &=\left\{x\in\mathcal{O}_{s,f}^{\Psi}\left|\begin{array}{c}\mbox{For all }v_0,\ldots,v_{s+1}\in \ker(f^*\omega)_x,\, \tilde{\tau}^{\Psi,2}_{(df)_x,(d^2f)_x}(v_0,v_1,v_2)=0 \mbox{ and } \\  (\mathcal{F}_{\Psi}\tilde{\tau}^{\Psi,m})_{(df)_x,(d^2f)_x,\ldots,(d^{m+1}f)_x}(v_0,v_1,\ldots,v_{m+1})=0 \mbox{ for }3\leq m+1\leq s+1\end{array}\right.\right\}.\label{osf}\end{align}  Recalling that, for $3\leq m+1\leq s+1$, we have by definition \[ \tilde{\tau}^{\Psi,m+1}_{(df)_x,(d^2f)_x,\ldots,(d^{m+1}f)_x}=\Pi(\mathcal{F}_{\Psi}\tilde{\tau}^{\Psi,m})_{(df)_x,(d^2f)_x,\ldots,(d^{m+1}f)_x}\] where $\Pi\co Mult^{m+2}(\R^d,\R^{2n})\to \mathcal{T}_{m+1}(\R^d,\R^{2n})$ is a projection as in Lemma \ref{keyproj}, the inclusion (\ref{us+1}) immediately follows from (\ref{osf}) and the definition of $\mathcal{K}_{2n-d,s+1}^{\Psi}(\R^d,\R^{2n})$.
\end{proof}

We now set about globalizing these results.

\begin{dfn}\label{mantrn} Let $X$ be a smooth $d$-dimensional manifold and let $(M,\omega)$ be a $2n$-dimensional symplectic manifold.  A $C^{r+1}$ map $f\co X\to M$ will be called \emph{transversely $r$-noncoisotropic} provided that there is an atlas $\{\phi_{\alpha}\co U_{\alpha}\to \R^d|\alpha\in A\}$ for $X$ and a collection of Darboux charts $\psi_{\alpha}\co V_{\alpha}\to\R^{2n}$ for $M$ such that $f(U_{\alpha})\subset V_{\alpha}$ and $\psi_{\alpha}\circ f\circ\phi_{\alpha}^{-1}\co \phi_{\alpha}(U_{\alpha})\to \R^{2n}$ is transversely $r$-noncoisotropic in the sense of Definition \ref{euctrn}.
\end{dfn}

For $a\in \mathbb{N}$ let $Imm^a(X,M)$ denote the space of $C^a$ immersions from $X$ to $M$.

\begin{lemma}\label{gentrn} For any  $a>r$ with $a\in \mathbb{N}\cup\{\infty\}$ the set \[ \{f\in C^a(X,M)|f\mbox{ is transversely $r$-noncoisotropic}\}\] is residual in the strong $C^a$ topology on $C^a(X,M)$.  If additionally $X$ is compact then  \[  \{ f\in Imm^a(X,M)| f\mbox{ is transversely $r$-noncoisotropic}\} \] is open in the $C^{r+1}$ topology on $Imm^a(X,M)$.
\end{lemma}

\begin{proof} Choose a countable atlas $\{\phi_i\co \mathcal{O}_i\to \R^d\}$ for $X$ such that the $\mathcal{O}_i$ form a basis for the topology on $X$.  Likewise choose a countable Darboux atlas $\{\psi_j\co V_j\to \R^{2n}\}$ such that the $V_j$ form a basis for the topology on $M$.  These atlases induce atlases on the jet manifolds $J^s(X,M)$, giving diffeomorphisms $\alpha_{ij}^{s}\co J^s(\mathcal{O}_i,V_j)\to J^s(\phi_i(\mathcal{O}_i),\psi_j(V_j))$ of the open sets $J^s(\mathcal{O}_i,V_j)\subset J^s(X,M)$ and 
$J^s(\phi_i(\mathcal{O}_i),\psi_j(V_j))\subset J^s(\mathbb{R}^d,\R^{2n})$.  As in the proof of Lemma \ref{gentrn0}, let $\{\Psi_k|k\in \mathbb{N}\}$ be a family of maps $\Psi_k\co U_{\Psi_k}\times \R^d\to \R^d$ as in Addendum \ref{add} such that the $U_{\Psi_k}$ cover $\mathcal{K}_{2n-d}$.  Now for $1\leq s\leq r$ and $i,j,k\in\mathbb{N}$ let \[ Z_{ijk}^{s}=(\alpha_{ij}^{s})^{-1}(\Phi_{2n-d}^{s,\Psi_k})\subset J^s(\mathcal{O}_i,V_j)\subset J^s(X,M).\]   It follows from the definitions that a $C^a$ map $f\co X\to M$ will be transversely $r$-noncoisotropic if for each $s=1,\ldots,r$ and each $i,j,k\in\mathbb{N}$ it holds that $j^s f$ is transverse  to $Z_{ijk}^{s}$.  But by the jet transversality theorem the set of $f$ having this latter property (for any given $i,j,k,s$) is residual in the strong $C^a$ topology, and so since a countable intersection of residual sets is residual we have proven the first sentence of the lemma.

We now assume that $X$ is compact and that $f\co X\to M$ is a transversely $r$-noncoisotropic immersion.  Using the compactness of $X$ we can find a finite atlas $\{\phi_i\co \mathcal{O}_i\to \R^d:i=1,\ldots,p\}$ for $M$ and a finite collection of Darboux charts $\psi_i\co V_i\to \R^{2n}$ so that each $f(\mathcal{O}_i)\subset V_i$ and each $\psi_i\circ f\circ \phi_{i}^{-1}$ is transversely $r$-noncoisotropic in the sense of Definition \ref{mantrn}.  Moreover we can arrange for there to be compact subsets $L_i\subset\mathcal{O}_i$ so that the $L_i$ still cover $M$. For each $i$ the image of $\phi(L_i)$ under $d(\psi_i\circ f\circ \phi_{i}^{-1})$ will then be covered by finitely many  open sets $U_{\Psi_{ik}}\subset Mon(\R^d,\R^{2n})$ with associated maps $\Psi_{ik}\co U_{\Psi_{ik}}\times \R^d\to\R^d$ such that $j^s(\psi_i\circ f\circ\phi_{i}^{-1})\pitchfork \Phi_{2n-d}^{s,\Psi_{ik}}$ for each $i,k$ and each $s\in\{1,\ldots,r\}$.  But a sufficiently $C^{r+1}$-small perturbation $\tilde{f}$ of $f$ will continue to have the properties that $\tilde{f}(L_i)\subset V_i$ and that $j^s(\psi_i\circ \tilde{f}\circ \phi_{i}^{-1})$ is transverse to  $\Phi_{2n-d}^{s,\Psi_{ik}}$ throughout $L_i$ (and therefore also throughout small neighborhoods $\mathcal{O}'_i$ of $L_i$, which will still be domains of the charts of an atlas for $X$).  Consequently $\tilde{f}$ will still be transversely $r$-non-coisotropic provided that $\tilde{f}$ is sufficiently $C^{r+1}$-close to $f$.
\end{proof}

\begin{prop}\label{trnwt} Assuming that the dimensions $d$ and $2n$ of, respectively, $X$ and $M$ obey $d\leq 2n-2$,  there is a number $r$, depending only on $d$ and $2n$, such that for any $C^{\infty}$ embedding $f\co X\to M$ which is transversely $r$-noncoisotropic and has closed image, the image $N=f(X)$ is weightless.  

Specifically, $r$ may be taken to be any positive integer such that \begin{equation}\label{rform} \sum_{s=2}^{r}\dim\mathcal{T}_s(\R^{2n-d},\R^{2n})>d-\frac{(2n-d+1)(2n-d)}{2}.\end{equation}
\end{prop}
\begin{remark} Recall that $\mathcal{T}_s(\R^{2n-d},\R^{2n})$ is the space of $(s+1)$-linear maps from $\R^{2n-d}$ to $\R^{2n}$ obeying properties (i)-(iii) from Lemma \ref{tauprop}.  Of course $\dim\mathcal{T}_s(\R^{2n-d},\R^{2n})$ can be computed, but the formula that results is not particularly illuminating; suffice it to note that, provided $2n-d\geq 2$, we have $\dim\mathcal{T}_s(\R^{2n-d},\R^d)\geq s$, since if $v_0,v_1\in \R^{2n-d}$ are linearly independent of each other then the values $\tau(v_0,\ldots,v_0,v_1,\ldots,v_1)$, where $v_0$ is repeated some number $i$ of times where $1\leq i\leq s$, are independent of each other as $i$ varies.  So it is  sufficient to take $r$ so that $\frac{r(r+1)}{2}-1>d-\frac{(2n-d+1)(2n-d)}{2}$. \end{remark}
\begin{proof}  Let $f\co X\to M$ be a transversely $r$-noncoisotropic embedding and $N=f(X)$. For any $x\in X$ the subspace $T_{f(x)}N^{\omega}\leq T_{f(x)}M$ has dimension $2n-d$, and $\ker(f^*\omega)_x=f_{*}^{-1}(T_{f(x)}N^{\omega}\cap T_{f(x)}N)$ has the same dimension as $T_{f(x)}N^{\omega}\cap T_{f(x)}N$.  Thus $T_{f(x)}N^{\omega}\leq T_{f(x)}N$ if and only if $\dim\ker(f^*\omega)_x=2n-d$. Thus, where \[ X_1=\{x\in X|\dim\ker(f^*\omega)_x=2n-d\},\]  Lemma \ref{rnp} (with $\mathcal{O}=P=N$) shows that $f^{-1}(R_N)\subset X_1$. 

By definition, the fact that $f$ is transversely $r$-noncoisotropic means that for each $x_0\in X_1$ there is a chart $\phi\co \mathcal{O}\to\R^d$ for $X$ around $x_0$, a Darboux chart $\psi\co V\to \R^{2n}$ for $M$ with $f(\mathcal{O})\subset V$, and a map $\Psi\co U_{\Psi}\times\R^d\to \R^d$ as in Addendum \ref{add} such that $d(\psi\circ f\circ\phi^{-1})_{\phi(x)}\in U_{\Psi}$ for each $x\in\mathcal{O}$, with the property that $j^s(\psi\circ f\circ\phi^{-1})\pitchfork \Phi_{2n-d}^{s,\Psi}$ for each $s=1,\ldots,r$.  By  Lemma \ref{RNlemma}(iii), to prove the proposition it suffices to show that (provided  $r$ obeys (\ref{rform})) $R_N=\varnothing$.  Since we have already established that $f^{-1}(R_N)\subset X_1$, it thus suffices to show that, for all data $\phi,\psi,\Psi$ as just described, the intersection $R_N\cap f(\mathcal{O})$ is empty.

To do this, define, for $s=1,\ldots,r$, \[ X_{s}^{\Psi}=\phi^{-1}\left(j^s(\psi\circ f\circ \phi^{-1})^{-1}(\Phi_{2n-d}^{s,\Psi})\right).\]  So evidently $X_{1}^{\Psi}=X_1\cap \mathcal{O}$, so $f^{-1}(R_N)\cap \mathcal{O}\subset X_{1}^{\Psi}$. Of course just as in the proof of Lemma \ref{trn} we have $X_{s+1}^{\Psi}\subset X_{s}^{\Psi}$ for $1\leq s\leq r-1$.  Since $\phi\co \mathcal{O}\to \R^{d}$ is a coordinate chart (and so a diffeomorphism to its image), it follows directly from Lemma \ref{trn} that, for $1\leq s\leq r-1$, \[ \{x\in X_{s}^{\Psi}|\ker(f^{*}\omega)_x\subset T_x X_{s}^{\Psi}\}\subset X_{s+1}^{\Psi}.\]  Now as follows from the discussion in the first paragraph of the proof, for $x\in  X_{1}^{\Psi}$ we have $f_*\left(\ker (f^{*}\omega)_x\right)=T_{f(x)}N^{\omega}$; thus \[ \{y\in f(X_{s}^{\Psi})|T_y N^{\omega}\subset T_y f(X_{s}^{\Psi})\}\subset f(X_{s+1}^{\Psi}).\]  So (again using that $f$ is an embedding) Lemma \ref{rnp} shows that if $R_N\cap f(\mathcal{O})\subset f(X_{s}^{\Psi})$, then $R_N\cap f(\mathcal{O})\subset f(X_{s+1}^{\Psi})$.  So since we have already shown that $R_N\cap f(\mathcal{O})\subset f(X_{1}^{\Psi})$, it follows by induction that $R_N\cap f(\mathcal{O})\subset f(X_{r}^{\Psi})$.

Now the codimension of $X_{r}^{\Psi}$ in $X$ is the same as the codimension of $\Phi_{2n-d}^{r,\Psi}$ in $J^r(\phi(\mathcal{O}),\R^{2n})$, which by Proposition \ref{KCR} is equal to $\frac{(2n-d)(2n-d-1)}{2}+\sum_{s=2}^{r}\dim\mathcal{T}_s(\R^{2n-d},\R^{2n})$.  

Now if $r$ has been chosen so that $\sum_{s=2}^{r}\dim\mathcal{T}_s(\R^{2n-d},\R^{2n})>d-\frac{(2n-d+1)(2n-d)}{2}$, then \begin{align*} \dim X_{r}^{\Psi}&=d-\frac{(2n-d)(2n-d-1)}{2}-\sum_{s=2}^{r}\dim\mathcal{T}_s(\R^{2n-d},\R^{2n})
\\&<\frac{(2n-d+1)(2n-d)}{2}-\frac{(2n-d)(2n-d-1)}{2}=2n-d,
\end{align*}
and so for any $y\in f(X_{r}^{\Psi})$ we have $\dim T_y N^{\omega}>\dim T_y f(X_{r}^{\Psi})$.  Thus \[ \{y\in f(X_{r}^{\Psi})|T_yN^{\omega}\leq T_yf(X_{r}^{\Psi})\}=\varnothing,\] and so a final application of Lemma \ref{rnp} (now with $P=f(X_{r}^{\Psi})$) shows that $R_N\cap f(\mathcal{O})=\varnothing$, which as noted earlier suffices to prove the result.
\end{proof}

We conclude with the following immediate consequence of Lemma \ref{gentrn} and Proposition \ref{trnwt}, which finally proves Theorem \ref{int1}(ii):

\begin{cor}\label{wtmain} If $(M,\omega)$ is a $2n$-dimensional symplectic manifold and if $X$ is a  manifold of dimension $d\leq 2n-2$ then where $Emb_C(X,M)$ denotes the space of closed $C^{\infty}$ embeddings of $X$ into $M$, \[ \{f\in Emb_C(X,M)|f(X)\mbox{ is weightless}\}\] contains a subset which is residual and hence dense in the strong $C^{\infty}$ topology on $Emb_C(X,M)$ and which, if $X$ is compact, is open in the strong $C^{r+1}$ topology provided that $r\geq 1$ obeys 
(\ref{rform}).
\end{cor}
\begin{proof} Indeed the subset may be taken to be the collection of transversely $r$-noncoisotropic embeddings by Proposition \ref{trnwt}.  The second part of Lemma \ref{gentrn} directly implies that this subset is open in the $C^{r+1}$ (and hence also the $C^{\infty}$) topology on $Emb_C(X,M)$ if $X$ is compact.    Meanwhile $Emb_C(X,M)$ is (whether or not $X$ is compact) open in  the strong $C^{\infty}$ topology  (\cite[Corollary 2.1.6]{Hi}) and so the first part of Lemma \ref{gentrn} together with the fact that the intersection of a residual set in $C^{\infty}(X,M)$ with any open subset $W$ of $C^{\infty}(X,M)$ is residual in the subspace topology on $W$ implies that our subset is residual in $Emb_{C}(X,M)$.
\end{proof}

\bigskip
\footnotesize
\noindent\textit{Acknowledgments.}
I am grateful to the anonymous referee for his/her detailed suggestions and corrections.  This research was partly supported by NSF Grant DMS-1105700.


\begin{thebibliography}{MMOPR}

\bibitem[AS]{AS} Abouzaid, M., Seidel, P.: \emph{An open string analogue of Viterbo functoriality}. Geom. Topol. \textbf{14} (2010), 627--718. 
\bibitem[AL]{ALP} Audin, M. Lafontaine, J., eds.: \emph{Holomorphic curves in symplectic geometry}. Progr. Math. \textbf{117}. Birkh\"auser, Basel, 1994.
\bibitem[Ba]{Ban}  Banyaga, A.: \emph{Sur la structure du groupe des diff\'eomorphismes qui pr\'eservent une forme symplectique}. Comment. Math. Helv. \textbf{53} (1978), 174--227.
\bibitem[BC]{BC06} Barraud, J.-F.,  Cornea, O.: \emph{Homotopical dynamics in symplectic topology}, in \emph{Morse theoretical methods in
non-linear analysis and symplectic topology}, Springer (2006), 109--148.
\bibitem[Ch98]{Ch98}  Chekanov, Yu.: \emph{Lagrangian intersections, symplectic energy, and areas of holomorphic curves}. Duke Math. J. \textbf{95} (1998), 213--226. 
\bibitem[Ch00]{Ch00} Chekanov, Yu.: \emph{Invariant Finsler metrics on the space of Lagrangian embeddings}. Math. Z. \textbf{234} (2000), 605--619. 
\bibitem[CGK]{CGK} Cieliebak, K.,   Ginzburg, V., Kerman, E.: \emph{Symplectic homology and periodic orbits near symplectic submanifolds}. Comment. Math. Helv. \textbf{79} (2004), 554--581. 
\bibitem[Fl]{Fl} Floer, A.: \emph{Morse theory for Lagrangian intersections}. J. Differential Geom. \textbf{28} (1988), 513--547.
\bibitem[FOOO]{FOOO}  Fukaya, K.,  Oh, Y.-G., Ohta, H., Ono, K.: \emph{Lagrangian Intersection Floer Theory: Anomaly and Obstruction}. 2 vols. AMS, Providence, 2009. 
\bibitem[Gi]{Gi}  Ginzburg, V.: \emph{Coisotropic intersections}. Duke Math. J. \textbf{140} (2007), 111--163.
\bibitem[Gu]{Gu}  G\"urel, B.: \emph{Totally non-coisotropic displacement and its applications to Hamiltonian dynamics}. Commun. Contemp. Math. \textbf{10} (2008), 1103--1128. 
\bibitem[Hi]{Hi}  Hirsch, M.: \emph{Differential Topology}. Grad. Texts. Math. \textbf{33}.  Springer-Verlag, New York, 1976. 
\bibitem[Ho]{Ho}  Hofer, H.: \emph{On the topological properties of symplectic maps}. Proc. Roy. Soc. Edinburgh Sect. A \textbf{115} (1990), 25--38. 
\bibitem[Kh09]{Kh}  Khanevsky, M.: \emph{Hofer's metric on the space of diameters}. J. Topol. Anal. \textbf{1} (2009), 407--416. 
\bibitem[LM]{LM}  Lalonde, F.,  McDuff, D.: \emph{The geometry of symplectic energy}. Ann. of Math. (2) \textbf{141} (1995), 349--371. 
\bibitem[LS]{LS} Laudenbach, F., Sikorav, J.-C.: \emph{Hamiltonian disjunction and limits of Lagrangian submanifolds}. Internat. Math. Res. Notices \textbf{1994},  161--168. 
\bibitem[Mah]{Ma64}  Mahowald, M.: \emph{On the normal bundle of a manifold}. Pacific J. Math. \textbf{14} (1964) 1335--1341.
\bibitem[Mar]{Ma} Marle, C.-M.: \emph{Sous-vari\'et\'es de rang constant d'une vari\'et\'e symplectique}.  In \emph{Third Schnepfenried geometry conference, Vol. 1 (Schnepfenried, 1982)}. Asterisque, 107--108. Soc. Math. France, Paris, 1983, 69--86. 
\bibitem[Mas]{Mas} Massey, W.: \emph{Normal vector fields on manifolds}. Proc. Amer. Math. Soc. \textbf{12} 1961, 33--40. 
\bibitem[MMOPR]{MMOPR}  Marsden, J.,  Misiolek, G.,  Ortega, J.-P., Perlmutter, M.,  Ratiu, T.: \emph{Hamiltonian reduction by stages}. Lecture Notes in Mathematics, \textbf{1913}. Springer, Berlin, 2007. 
\bibitem[MS]{MS}  McDuff, D., Salamon, D.: \emph{Introduction to symplectic topology}. Oxford Mathematical Monographs.
Oxford University Press, Oxford, 1998. 
\bibitem[Oh93]{Oh93} Oh, Y.-G.: \emph{Floer cohomology of Lagrangian intersections and pseudo-holomorphic disks. I}.  Comm. Pure Appl. Math. \textbf{43} (1993), no. 7, 949--993.
\bibitem[Oh97a]{Oh1}  Oh, Y.-G.: \emph{Symplectic topology as the geometry of action functional. I. Relative Floer theory on the cotangent bundle}. J. Differential Geom. \textbf{46} (1997),  499--577. 
\bibitem[Oh97b]{Oh} Oh, Y.-G.:  \emph{Gromov-Floer theory and disjunction energy of compact Lagrangian embeddings}. Math. Res. Lett. \textbf{4} (1997),  895--905. 
\bibitem[Os03]{Os}  Ostrover, Y.: \emph{A comparison of Hofer's metrics on Hamiltonian diffeomorphisms and Lagrangian submanifolds}. Commun. Contemp. Math. 5 (2003), 803--811.  
\bibitem[P91]{P91}  Polterovich, L.: \emph{The surgery of Lagrange submanifolds}. Geom. Funct. Anal. \textbf{1} (1991), 198--210. 
\bibitem[P95]{P95} Polterovich, L.:  \emph{An obstacle to non-Lagrangian intersections}. In \emph{The Floer memorial volume}, Progr. Math., \textbf{133}, Birkhäuser, Basel, 1995, 575--586.
\bibitem[SL]{SL} Sjamaar, R., Lerman, E.: \emph{Stratified symplectic spaces and reduction}. Ann. Math. \textbf{134} (1991), no. 2, 375--422.
\bibitem[S]{St} Steenrod, N.: \emph{The Topology of Fibre Bundles}. Princeton Mathematical Series, \textbf{14}. Princeton University Press, Princeton, N. J., 1951. 
\bibitem[U11a]{U}  Usher, M.: \emph{Boundary depth in Floer theory and its applications to Hamiltonian dynamics and coisotropic submanifolds}. Israel J. Math. \textbf{184} (2011), 1--57. 
\bibitem[U11b]{U2} Usher, M.: \emph{Hofer's metrics and boundary depth}. arXiv:1107.4599. 
\bibitem[Zi]{Zi} Ziltener, F.: \emph{Coisotropic submanifolds, leaf-wise fixed points, and presymplectic embeddings}.  J. Symplectic Geom. \textbf{8} (2010), 95–-118.  
\end{thebibliography}
\end{document}